\documentclass[a4paper,12pt]{amsart}
\usepackage{amsmath,amssymb,amsthm}
\usepackage{amscd}
\usepackage{array,enumerate}
\newtheorem{Thm}{Theorem}[section]
\newtheorem{Prop}[Thm]{Proposition}
\newtheorem{Lem}[Thm]{Lemma}
\newtheorem{Cor}[Thm]{Corollary}
\newtheorem{Thmint}{Theorem}[section]
\newtheorem{Corint}[Thmint]{Corollary}

\theoremstyle{definition}
\newtheorem{Rem}[Thm]{Remark}

\newtheorem{Def}[Thm]{Definition}
\newtheorem{Exm}[Thm]{Example}
\newtheorem{Exms}[Thm]{Examples}

\newtheorem{Prob}[Thm]{Problem}
\newcommand{\Cs}{\mbox{${\rm C}^\ast$}}
\newcommand{\id}{\mbox{\rm id}}
\newcommand{\GL}{\mbox{\rm GL}}

\newcommand{\ad}{\mbox{\rm ad}}
\title[Amenable minimal Cantor systems]{Amenable minimal Cantor systems of free groups arising from diagonal actions}
\author{Yuhei Suzuki}
\subjclass[2000]{Primary~37B05, Secondary~46L80, 54H20}
\keywords{\Cs -algebras; amenable actions; free groups; Cantor systems.}
\address{Department of Mathematical Sciences,
University of Tokyo, Komaba, Tokyo, 153-8914, Japan}
\email{suzukiyu@ms.u-tokyo.ac.jp}
\begin{document}
\begin{abstract}
We study amenable minimal Cantor systems of free groups
arising from the diagonal actions of the boundary actions and certain Cantor systems.
It is shown that every virtually free group
admits continuously many amenable minimal Cantor systems
whose crossed products are mutually non-isomorphic Kirchberg algebras in the UCT class
(with explicitly determined $K$-theory).
The technique developed in our study also enables us to compute the $K$-theory of
certain amenable minimal Cantor systems.
We apply it to the diagonal actions of
the boundary actions and the products of the odometer transformations,
and determine their $K$-theory.
Then we classify them in terms of the topological full groups, continuous orbit equivalence,
strong orbit equivalence, and the crossed products.
\end{abstract} 
\maketitle
\tableofcontents

\section{Introduction and main results}
\subsection{Introduction}\label{ss:Int}
The subject of this paper is amenable minimal Cantor systems of free groups.
The Cantor set is characterized by the following four properties: compactness, total disconnectedness,
metrizability, and not having isolated points.
From this characterization, the property `$X$ is (homeomorphic to) the Cantor set' is preserved by many
operations. For example, it is preserved by taking a finite direct sum, a countable direct product, and
the projective limit of a sequence with surjective connecting maps.
Moreover, the properties which characterize the Cantor set make topological difficulties small in many situations.
By these properties, the Cantor set can be considered as a topological analogue of the Lebesgue space without atoms.
Moreover, in the category of minimal dynamical systems on metrizable compact spaces,
the Cantor set has a ``universal'' property in the following sense.
For any minimal dynamical system of a countable infinite group $\Gamma$ on a metrizable compact space,
it is realized as a quotient of a minimal Cantor $\Gamma$-system.
(Here and throughout this paper, we call a dynamical system on the Cantor set as a Cantor system.)
This follows from a similar proof to the case $\Gamma=\mathbb{Z}$; see \cite[Section 1]{GPS}.
Therefore the study of minimal Cantor systems is important.
Furthermore, Cantor systems themselves are attractive objects.
The underlying spaces of many important dynamical systems are homeomorphic to the Cantor set.
This includes certain symbolic dynamical systems, the boundary actions of virtually free groups, and the odometer transformations.

The free group $\mathbb{F}_n$ is one of the most interesting and tractable non-amenable groups.
Most of known non-amenable groups contain $\mathbb{F}_n$, and it has nice properties: the universal property (namely, the freeness), exactness, the Haagerup property, weak amenability,
hyperbolicity (with the nice boundary), and so on.
Hence, to understand the phenomena of non-amenable groups,
the free groups are suitable objects for the first study.

The aim of this paper is to construct and study amenable minimal Cantor systems of free groups.
This is motivated by the following two natural questions.
The first is finding new concrete and tractable presentations of Kirchberg algebras in the UCT class,
which is asked in the book \cite{Ror} of R\o rdam. (See the last paragraph of page 85.)
The second is that how well the crossed products of amenable minimal Cantor $\mathbb{F}_n$-systems remember
the information about the original systems.

Note that for the case of the group $\mathbb{Z}$, analogues of both questions have complete answers.
They are the celebrated results of Giordano, Putnam, and Skau \cite{GPS}.
For the first question, they have shown that every simple unital A$\mathbb{T}$-algebra of real rank zero whose $K_1$-group is isomorphic to $\mathbb{Z}$ is presented as the crossed product
of a minimal Cantor $\mathbb{Z}$-system and this is the only possible case.
For the second question, they have shown that two minimal Cantor $\mathbb{Z}$-systems
have isomorphic crossed products
exactly when they are strongly orbit equivalent.

To start the study on these problems, we need well-understandable examples of amenable minimal Cantor $\mathbb{F}_n$-systems as many as possible.
Until now, only a few examples have been constructed and studied.
In this paper, we construct continuously many examples of amenable minimal Cantor $\mathbb{F}_n$-systems whose crossed products are completely determined.
As a consequence, for the first question, we obtain new concrete presentations for certain continuously many Kirchberg algebras in the UCT class.
For the second one, our examples give a hopeful prospect.
As examples, we show that the diagonal actions of the boundary actions and the products of odometer
transformations are classified in terms of continuous orbit equivalence by using a ${\rm C}^{\ast}$-algebraic technique.

A recent work of G. A. Elliott and A. Sierakowski \cite{ES} gives us an example of
amenable minimal Cantor $\mathbb{F}_n$-systems which are distinguished by $K$-theory.
They construct an amenable minimal free Cantor $\mathbb{F}_n$-system
whose $K_0$-group vanishes.
In particular, it has the different $K_0$-group from that of the boundary action.
Their construction is based on the idea developed in the paper \cite{RS}.
Our strategy is different from that of \cite{ES}.
We construct amenable minimal Cantor $\mathbb{F}_n$-systems from the diagonal actions.
This construction is quite simple and gives many fruitful and concrete examples of amenable minimal Cantor $\mathbb{F}_n$-systems.

\subsection{Main results}
Here we collect the main results of this paper.

Throughout the paper, the $K$-theory of the reduced crossed product of a dynamical system $\gamma$
is referred to as the $K$-theory of $\gamma$ for short.
\begin{Thmint}[Theorem \ref{Thm:Inf}]\label{Thmint:Inf}
Let $G$ be a subgroup of $\mathbb{Q}^{\oplus\infty}$ which contains $\mathbb{Z}^{\oplus\infty}$ as a subgroup
of infinite index.
Let $2\leq n<\infty$ and $k$ be an integer.
Then there exists an amenable minimal Cantor $\mathbb{F}_n$-system that satisfies the following properties.
\begin{itemize}
\item The pair of $K_0$-group and the unit $[1]_0$ is isomorphic to
\[\left(G\oplus\Lambda_{G, n}, 0\oplus [k(n-1)^{-1}]\right),\]
where $\Lambda_{G,n}$ is the subgroup of $\mathbb{Q}/\mathbb{Z}$ consisting of
elements whose order divides the product of $(n-1)$ and the order of a finite subgroup of $G/\mathbb{Z}^{\oplus\infty}$.
\item The $K_1$-group is isomorphic to $\mathbb{Z}^{\oplus \infty}$.
\item The crossed product is a Kirchberg algebra in the UCT class.
\end{itemize}
\end{Thmint}
As a consequence, we obtain the following result.
\begin{Corint}[Corollary \ref{Cor:Inf}]
Every free group admits continuously many amenable minimal Cantor systems
which are distinguished by the $K$-groups.
\end{Corint}
We also show similar results for (finitely generated, non-amenable) virtually free groups (Theorem \ref{Thm:Infv}).
Then, as a consequence of these results, we obtain the following decomposition theorem.
\begin{Corint}[Corollary \ref{Cor:Dec}]
For a torsion free abelian group $G$ of infinite rank,
consider a Kirchberg algebra $A$ in the UCT class satisfying
$(K_0(A), [1]_0, K_1(A))\cong(G\oplus\mathbb{Q}/\mathbb{Z}, 0, \mathbb{Z}^{\oplus\infty})$.
Then for any virtually free group $\Gamma$, $A$ is decomposed
as the crossed product of an amenable minimal topologically free Cantor $\Gamma$-system.
\end{Corint}

We also see that even we restrict our attention to the free Cantor systems,
we still obtain the existence of continuously many amenable minimal Cantor systems.
\begin{Thmint}[Theorem \ref{Thm:free}]
Let $\Gamma$ be a virtually free group.
Then there exist continuously many amenable minimal free Cantor $\Gamma$-systems
whose crossed products are mutually non-isomorphic Kirchberg algebras in the UCT class.
\end{Thmint}
In the proof of Theorem \ref{Thmint:Inf},
techniques of the computation of $K$-theory are developed for certain Cantor systems.
In Section \ref{sec:odo}, we give computations of the $K$-theory for the diagonal actions of the boundary actions and the products of the odometer transformations.
From our computations, their topological full groups, continuous orbit equivalence classes,
and strong orbit equivalence classes are classified.
Here we collect the classification results.
First we present the Cantor systems which we will classify more precisely.
For each free group $\mathbb{F}_n$, fix an enumeration $\{s_1, \ldots, s_n\}$ of the canonical generators.
For $2\leq n<\infty$, $1\leq k\leq n$,
and a sequence $N_1, \ldots, N_k$ of infinite supernatural numbers,
define a Cantor $\mathbb{F}_n$-system by
\[\gamma_{N_1,\ldots, N_k}^{(n)}:=\beta_n\times\left(\prod_{j=1}^k\alpha_{N_j}\circ \pi_j^{(n)}\right),\]
where $\beta_n$ denotes the boundary action of $\mathbb{F}_n$,
$\alpha_{N}$ denotes the odometer transformation of type $N$, and
$\pi_j^{(n)}$ denotes the homomorphism
$\mathbb{F}_n\rightarrow \mathbb{Z}$
given by $s_j\mapsto 1$ and $s_i \mapsto 0$ for $i\neq j$.
\begin{Thmint}[Proposition \ref{Prop:grp} and Theorem \ref{Thm:Cla}]
For two Cantor systems $\gamma_1:=\gamma_{N_1,\ldots, N_k}^{(n)}$ and
$\gamma_2:=\gamma_{M_1,\ldots, M_l}^{(m)}$ defined as above, the following conditions are equivalent.
\begin{enumerate}[\upshape (1)]
\item They are strongly orbit equivalent.
\item They are continuously orbit equivalent.
\item Their topological full groups are isomorphic.
\item The commutator subgroups of their topological full groups are isomorphic.
\item Their crossed products are isomorphic.
\item Their $K_0$-invariants $(K_0, [1]_0)$
are isomorphic.
\item The equations $k=l$ and $n=m$ hold and there exist
a permutation $\sigma\in\mathfrak{S}_k$ and sequences
$(n_1, \ldots, n_k)$ and $(m_1, \ldots, m_k)$ of natural numbers
that satisfy $\prod_{j=1}^k n_j=\prod_{j=1}^k m_j$ and $n_iN_i=m_iM_{\sigma(i)}$.
\end{enumerate}
\end{Thmint}

\section{Preliminaries}
In this section, we collect the fundamental knowledges and notation used in this paper.
The basic references are the book \cite{BO} of N. Brown and N. Ozawa and the book \cite{Ror} of M. R\o rdam.
\subsection{Minimality of dynamical systems}
The minimality of topological dynamical systems is the indecomposability condition of topological dynamical systems.
It is regarded as a topological analogue of ergodicity.
Hence it is natural and important to study minimal dynamical systems in particular.
Here we recall the definition of minimal dynamical system.
\begin{Def}
Let $\Gamma$ be a group, $X$ be a topological space,
and $\alpha \colon \Gamma\curvearrowright X$ be a topological dynamical system of $\Gamma$ on $X$.
The dynamical system $\alpha$ is said to be minimal if
every orbit of $\alpha$ is dense in $X$.
\end{Def}
It is clear from the definition that $\alpha$ is minimal
if and only if there is no nontrivial $\Gamma$-invariant open/closed subset of $X$.

Minimality with a certain additional assumption makes the reduced crossed product simple.
For example, Archbold and Spielberg \cite[page 122, Corollary]{AS} have shown that this is the case if the action is topologically free.
Recall that a topological dynamical system $\alpha\colon\Gamma\curvearrowright X$ is said to be topologically free
if for any element $g\in\Gamma\setminus\{e\}$,
the subset $\{x\in X: \alpha(g)(x)\neq x\}$ is dense in $X$.

\subsection{Amenability of dynamical systems}
The (topological) amenability of dynamical systems on a compact space
is a dynamical analogue of the amenability of discrete groups.
First we review the definition of topological amenability.
For the definition, we need the space
${\rm Prob}(\Gamma)$, which is the space of all probability measures on $\Gamma$
with the pointwise convergence topology. On ${\rm Prob}(\Gamma)$, $\Gamma$ acts from the left
by $s.\mu(t):=\mu(s^{-1}t)$ for $s, t\in\Gamma$ and $\mu\in{\rm Prob}(\Gamma)$.
\begin{Def}
A dynamical system $\alpha$ of a group $\Gamma$
on a compact space $X$ is said to be amenable if there is a sequence
$(\mu_n)_n$ of continuous
maps
\[\mu_n\colon x\in X\mapsto \mu_n^x\in {\rm Prob}(\Gamma)\]
such that for all $s\in\Gamma$,
$\sup_{x\in X}(\|s.\mu_n^x-\mu_n^{s.x}\|_1)$ converges to zero as $n$ tends to infinity.
\end{Def}
Note that by definition, if a quotient of a dynamical system is amenable,
then the original dynamical system itself is amenable.
From this, even in the category of amenable minimal dynamical systems on metrizable compact spaces,
the Cantor set has the ``universal'' property in the same sense as described in the first paragraph of Section \ref{ss:Int}.

Amenable dynamical systems arise naturally in many situations.
Here we review a few examples of amenable dynamical systems.
\begin{Exms}[See \cite{BO}]

\begin{itemize}
\item Any dynamical system of an amenable group is amenable.
\item The boundary action of a hyperbolic group is amenable.
\item For a second countable locally compact group $G$, a discrete subgroup $\Gamma$, and a closed cocompact amenable subgroup $\Lambda$, the left multiplication action of $\Gamma$ on $G/\Lambda$ is amenable.
\item The left translation action of $\Gamma$ on its Stone-\v{C}ech compactification $\beta\Gamma$
is amenable if and only if $\Gamma$ is exact.
\end{itemize}
\end{Exms}
Next we review some basic and important properties of amenable dynamical systems.
For an amenable dynamical system and
for any non-amenable subgroup $\Lambda$ of the acting group, there is no $\Lambda$-invariant probability measure.
It is easy to check that any amenable minimal dynamical system of $\mathbb{F}_n$ must be topologically free.
The most important feature of amenability for us is that it ensures that
the crossed product has nice properties.
For example, the reduced crossed product of an amenable dynamical system is nuclear
(in fact this characterize the amenability) \cite{Ana0}, satisfies the UCT \cite{Tu},
and coincides with the full crossed product \cite{Ana0}.
Thus to study dynamical systems of discrete groups on compact spaces by a \Cs -algebraic way, it is natural to restrict our attention to amenable dynamical systems.
For more information about applications and examples of amenable dynamical systems, we refer the reader to the survey paper \cite{Oz} of N. Ozawa and the references therein.

\subsection{Gromov boundary}
For a pair of a finitely generated group $\Gamma$ and a finite generating set $S$ of $\Gamma$, we equip a geodesic left-invariant metric $d_S$ on $\Gamma$ by $d_S(g, h):=|g^{-1}h|_S$ where $|\cdot |_S$ is the length function on $\Gamma$ determined by $S$.
The quasi-isometric class of $d_S$ is independent of the choice of $S$,
hence any quasi-isometric invariant property of (geodesic) metric spaces defines a property of finitely generated groups.
The hyperbolicity is one such property.
For each hyperbolic group $\Gamma$,
there is a canonical boundary $\partial\Gamma$, called the Gromov boundary,
which is a metrizable compact Hausdorff space.
Roughly speaking, $\partial\Gamma$ is the space of all equivalence
classes of infinite geodesic rays in $\Gamma$.
Every hyperbolic group has the canonical action on its Gromov boundary, called the boundary action.
We denote the boundary action by $\beta$.
The boundary actions are known to be amenable.
Here we only recall the following fundamental fact.
For the details, we refer the reader to \cite[Section 5.3]{BO} and the survey paper \cite{KB}.
\begin{Prop}
Let $\Gamma$ and $\Lambda$ be hyperbolic groups.
Let $f\colon \Gamma\rightarrow \Lambda$ be a group inclusion whose image is a finite index subgroup of $\Lambda$.
Then $f$ induces a $\Gamma$-equivariant homeomorphism
$\partial f$ from $\partial\Gamma$ onto $\partial\Lambda$.
\end{Prop}
A typical example of a hyperbolic group is the free group $\mathbb{F}_n$.
For the free group $\mathbb{F}_n$,
its Gromov boundary $\partial\mathbb{F}_n$ has a simple description.
This is because the Cayley graph of $\mathbb{F}_n$ with respect to the canonical generating set
is a tree.
Here we give it as a definition of the Gromov boundary of $\mathbb{F}_n$.
\begin{Def}
Let $S$ be the set of canonical generators of $\mathbb{F}_n$ and
set $\widetilde{S}:=S\sqcup S^{-1}.$
Define the subspace $\partial\mathbb{F}_n$ of $\prod_\mathbb{N}\widetilde{S}$ by
\[\partial\mathbb{F}_n:=\left\{(s_m)_{m\in\mathbb{N}}\in\prod_\mathbb{N}\widetilde{S}:
s_{m+1}\neq s_m^{-1} {\rm \ for\ all\ } m\in\mathbb{N}\right\}.\]
We equip $\partial\mathbb{F}_n$ with the topology induced from the product topology.
\end{Def}
It is easy to check that $\partial\mathbb{F}_n$ satisfies the four properties
which characterize the Cantor set, hence $\partial\mathbb{F}_n$ is homeomorphic to the Cantor set.

Each element of $\partial\mathbb{F}_n$ is regarded as a (one-sided) infinite reduced word of the free basis $S$.
For an element $w$ of $\mathbb{F}_n$ or $\partial\mathbb{F}_n$ with the reduced word $w=s_1\cdots s_k\cdots$,
the elements $s_1\cdots s_k$ and $s_k$ are referred to as the first $k$th segment of $w$ and
the $k$th alphabet of $w$, respectively.
For $w\in\mathbb{F}_n$ and $k\leq |w|=m$, the element 
$s_{m-k+1}\cdots s_m$ is referred to as the last $k$th segment of $w$.

For $z\in\mathbb{F}_n$ and $w\in\partial\mathbb{F}_n$, we define the product $z\cdot w$
by the same rule as that of the product of two elements of $\mathbb{F}_n$.
This is the boundary action of $\mathbb{F}_n$.
We denote the boundary action of $\mathbb{F}_n$ by $\beta_n$,
or simply by $\beta$ if the rank $n$ is obvious from the context.

As well the elements of free groups, for any other free basis $T$ of $\mathbb{F}_n$,
every element $w$ of $\partial\mathbb{F}_n$ can be expanded uniquely as an infinite reduced word of the free basis $T$.
This enables us to identify the boundary space $\partial\mathbb{F}_n$ with the space
\[\left\{(t_m)_{m\in\mathbb{N}}\in\prod_\mathbb{N}\widetilde{T}:
t_{m+1}\neq t_m^{-1} {\rm \ for\ all\ } m \in \mathbb{N}\right\},\]
where $\widetilde{T}:=T\sqcup T^{-1}$,
for any free basis $T$ of $\mathbb{F}_n$.
We always identify these spaces in this way without saying.

For a free basis $T$ of $\mathbb{F}_n$ and $t\in\widetilde{T}$,
we define the clopen subset $\Omega(t; T)$ of $\partial\mathbb{F}_n$
to be the subspace of all infinite reduced words whose first alphabet is $t$ in the expansion with respect to the free basis $T$.
When the free basis $T$ is obvious from the context,
we simply denote it by $\Omega(t)$.
More generally, for a free basis $W$ of a finite index subgroup $\Gamma$ of $\mathbb{F}_n$
and $w\in W\sqcup W^{-1}$,
we define the clopen subset $\Theta(w; W)$ of $\partial\mathbb{F}_n$
to be the image of $\Omega(w; W)(\subset\partial\Gamma)$ under the homeomorphism
$\partial\Gamma\cong\partial\mathbb{F}_n$ induced from the inclusion map.
If we need to refer the entire group $\Lambda=\mathbb{F}_n$,
we further denote it by $\Theta(w; W; \Lambda)$.

\subsection{Kirchberg algebras and UCT class}
A \Cs -algebra is said to be a Kirchberg algebra
if it is simple, separable, nuclear and purely infinite.
In this paper, we only deal the unital \Cs -algebras.
A \Cs -algebra is said to be in the UCT class if it satisfies
the universal coefficient theorem of Rosenberg--Schochet \cite{RSc}.

The celebrated theorem of Kirchberg \cite{Kir} and Phillips \cite{Phi} shows that two (unital) Kirchberg algebras $A_1$ and $A_2$ in the UCT class
are isomorphic if and only if the $K$-theoretic invariants $(K_0(A_i), [1]_0, K_1(A_i))$
are isomorphic. Namely, the isomorphism classes of Kirchberg algebras in the UCT class are completely determined
by their $K$-theory.
Typical examples of Kirchberg algebras in the UCT class are Cuntz algebras $\mathcal{O}_n$ ($2\leq n\leq\infty$) and
Cuntz--Krieger algebras $\mathcal{O}_A$ where $A$ is an irreducible $\{0, 1\}$-valued finite matrix that is not a permutation matrix.
It is known that for any countable abelian groups $G_0$, $G_1$
and $u\in G_0$, there is a Kirchberg algebra $A$ in the UCT class such that
$(K_0(A), [1]_0, K_1(A))$ is isomorphic to $(G_0, u, G_1)$ \cite[Theorem 3.6]{Rord}.

Kirchberg algebras in the UCT class naturally arise in many constructions of \Cs -algebras.
For example, certain graphs (see e.g., \cite{Rae}) and certain topological dynamical systems (see e.g., \cite{Ana}, \cite{LaS}, \cite{RoS}, \cite{RS}, and \cite{Spi}) provide Kirchberg algebras in the UCT class.
(However, in the most cases, the computations of the $K$-theory of the crossed products of dynamical systems
considered in these papers are hard and are still not known.
They are possible for the boundary actions of free products of finitely many cyclic groups \cite{Spi}.
See also \cite[Section 3]{LaS} and \cite{RoS} for computation results.)
A notable theorem of T. Katsura \cite{Kat} states that all Kirchberg algebras in the UCT class are realized as
a \Cs -algebra associated to a topological graph.
See also \cite{Kat2} for an explicit description of the Kirchberg algebras in the UCT class.
For precise information on the Kirchberg algebras, we refer the reader to the book \cite{Ror} of R\o rdam.

\subsection{Projective limit}
For a sequence $(X_n, \pi_n)_n$ of compact spaces $X_n$
with continuous maps $\pi_n\colon X_{n+1} \rightarrow X_n$,
the projective limit $\varprojlim(X_n, \pi_n)$ of the sequence $(X_n, \pi_n)_n$ is the compact space
defined by
\[\left\{(x_n)_{n\in\mathbb{N}}\in\prod_{n\in\mathbb{N}} X_n:\pi_n(x_{n+1})=x_n {\rm\ for\ all\ }n \in\mathbb{N}\right\}\]
(with the relative product topology).
It is easy to check that the above subspace is closed,
hence it is indeed a compact space.
For a group $\Gamma$,
if we have a sequence $(\alpha_n)_n$ of dynamical systems
$\alpha_n\colon \Gamma\curvearrowright X_n$
which is compatible with $(\pi_n)_n$ (namely $\pi_n\circ\alpha_{n+1}=\alpha_{n}\circ\pi_n$ holds for all $n$),
then the diagonal action of the sequence induces the action of $\Gamma$ on the projective limit.
We refer to this induced action as the projective limit of the projective system $(\alpha_n)_n$.
Note that by definition, the reduced crossed product $C(\varprojlim X_n)\rtimes\Gamma$
is isomorphic to the inductive limit $\varinjlim (C(X_n)\rtimes\Gamma)$.

(We remark that projective systems and projective limits are also defined for more general setting.
The index of the projective system can be arbitrary directed set.
However, in this paper, we only deal the projective systems of sequences.)

\subsection{Supernatural numbers and associated abelian groups}
To describe certain abelian groups, we need the notion of the supernatural numbers.
Denote by $\mathcal{P}$ the set of all prime numbers.
A supernatural number is a map
from $\mathcal{P}$ into the set $\{0, 1, \ldots, \infty\}$.
A supernatural number $N$ is formally presented as the formal infinite product
$\prod_{p\in\mathcal{P}} p^{N(p)}$ of powers of prime numbers.
By this presentation, supernatural numbers are naturally
regarded as a generalization of natural numbers.
We say a supernatural number $N$ is infinite if it is not a natural number.
Note that many operations of natural numbers are naturally extended to that of supernatural numbers.
(E.g., the (possibly infinite) product, the greatest common divisor, and the least common multiple; which
correspond to the summation, the infimum, and the supremum of the corresponding functions $N$, respectively.)

For a supernatural number $N$,
denote by $\Lambda(N)$ the subgroup of $\mathbb{Q}/\mathbb{Z}$ generated by
the elements whose order divides $N$ and
denote by $\Upsilon(N)$ the inverse image of the group $\Lambda(N)$ under the quotient homomorphism
$\mathbb{Q}\rightarrow\mathbb{Q}/\mathbb{Z}$.
Note that for two supernatural numbers $N$ and $M$,
the groups $\Lambda(N)$ and $\Lambda(M)$ are isomorphic
if and only if $N=M$ holds, and the groups $\Upsilon(N)$ and $\Upsilon(M)$ are isomorphic
if and only if there are natural numbers $n$ and $m$ such that the equality
$nN=mM$ holds.
\subsection*{Notation}
Here we fix notation which are used throughout this paper.
\begin{itemize}
\item Denote by $\mathbb{F}_n$ the free group of the rank $n$.
We always assume $2\leq n<\infty$.
\item We denote by $S$ the set of canonical generators of $\mathbb{F}_n$.
\item For $w\in\mathbb{F}_n$, $|w|$ denotes the length of the reduced word of $w$.
\item For $m\in \mathbb{N}$, denote by $\mathbb{Z}_m$ the cyclic group of order $m$.
\item For two actions $\sigma$ and $\tau$ of a group $\Gamma$, denote by $\sigma\times\tau$
the diagonal action of $\sigma$ and $\tau$.
\item For a group action $\alpha\colon \Gamma\curvearrowright X$, $g\in \Gamma$, $x\in X$, and $Y\subset X$,
we sometimes denote $\alpha(g)(x)$, $\alpha(g)(Y)$ by $g.x$, $g.Y$, respectively, for short.
\item For a subset $X$ of a group $\Gamma$,
$\langle X\rangle$ denotes the subgroup of $\Gamma$ generated by $X$.
\item For a group $\Gamma$, denote by $D(\Gamma)$ the commutator subgroup of $\Gamma$,
that is, the subgroup generated by elements of the form $ghg^{-1}h^{-1}, g, h\in\Gamma$.
\item In this paper we only consider the reduced crossed product.
For this reason, we simply refer to it as the crossed product, and represent it by the symbol `$\rtimes$'.
\item We denote by $\lambda(g)$ the canonical implementing unitary element of $g\in\Gamma$ in the crossed product $A\rtimes\Gamma$.
\item For a \Cs -algebra $A$ and a finite set $X$,
denote by $\mathbb{M}_X(A)$ the \Cs -algebra of all $A$-valued $X$ by $X$ matrices.
\end{itemize}

\section{Elementary construction}\label{sec:Cons}
In this section we give the construction of amenable minimal Cantor $\mathbb{F}_n$-systems
which plays the fundamental role in the next two sections.
The next theorem provides amenable minimal dynamical systems of hyperbolic groups.

\begin{Thm}\label{Thm:bou}
Let $\Gamma$ be a hyperbolic group.
Let $\sigma$ be a transitive action of $\Gamma$ on a finite set $X$.
Then the diagonal action $\beta\times\sigma$ is amenable and minimal,
and its crossed product is isomorphic to
$\mathbb{M}_n(C(\partial\Gamma_0)\rtimes\Gamma_0)$,
where $\Gamma_0$ is the stabilizer subgroup of $\sigma\colon\Gamma\curvearrowright X$ at a point,
which is independent of the choice of point up to conjugacy by transitivity,
and $n=\sharp X$. 
\end{Thm}
\begin{proof}
Amenability of $\beta\times\sigma$ is clear since it has an amenable quotient.
Fix $x_0\in X$ and denote by $\Gamma_0$ the stabilizer subgroup of $\sigma$ at $x_0$.
Then $\Gamma_0$ is a finite index subgroup of $\Gamma$,
and therefore it is hyperbolic and the restriction of the boundary action of $\Gamma$ to $\Gamma_0$
coincides with the boundary action of $\Gamma_0$.
(Recall that the inclusion $\Gamma_0\hookrightarrow\Gamma$
induces a $\Gamma_0$-equivariant homeomorphism of boundary spaces.)
This shows that the restriction of $\beta$ to $\Gamma_0$ is minimal,
hence $\beta\times\sigma$ is minimal.

The last claim immediately follows from Green's imprimitivity theorem \cite[Theorem 4.1]{Gre}.
However, we need a concrete isomorphism for later use, so we construct an isomorphism directly here,
which is a very special case of \cite{Gre}.
To do this, first we identify $X$ with $\Gamma/\Gamma_0$
by the bijective map $s\in\Gamma/\Gamma_0\mapsto s.x_0 \in X.$
Take a cross section $\rho$ of the quotient map
$\Gamma\rightarrow\Gamma/\Gamma_0$.
Then define two maps $\pi$ and $u$ by
\[\pi\colon f\in C(\partial\Gamma\times X)\mapsto \bigoplus_{x\in X} \left(f\circ((\beta\times\sigma)(\rho(x)))|_{\partial\Gamma\times \{x_0\}}\right)\in\mathbb{M}_X(C(\partial\Gamma_0)\rtimes\Gamma_0)\]
\[{\rm and}\]
\[u\colon s\in\Gamma \mapsto \sum_{x\in X} E_{s.x, x}\otimes \lambda(\rho(s.x)^{-1}s\rho(x))\in\mathbb{M}_X(C(\partial\Gamma_0)\rtimes\Gamma_0),\]
here we identify $\partial\Gamma\times \{x_0\}$ with $\partial\Gamma_0$ in the canonical way.
Then the pair $(\pi, u)$ is a covariant representation of $\beta\times\sigma$.
This covariant representation induces a $\ast$-isomorphism $\theta$ between two $\ast$-algebras
$C(\partial\Gamma\times X)\rtimes_{\rm alg}\Gamma$ and $\mathbb{M}_X(C(\partial\Gamma_0)\rtimes_{\rm alg}\Gamma_0)$
where the symbol `$\rtimes_{\rm alg}$' denotes the algebraic crossed product.
Then by amenability of $\beta\times\sigma$,
the universal \Cs -enveloping algebra of the $\ast$-algebra
$C(\partial\Gamma\times X)\rtimes_{\rm alg}\Gamma$ coincides with the reduced crossed product of $\beta\times\sigma$,
and similarly for the second one.
This shows that the $\ast$-isomorphism $\theta$ extends to the desired isomorphism.
\end{proof}
\begin{Rem}
In Theorem \ref{Thm:bou}, if the stabilizer subgroup $\Gamma_0$ is ICC,
then $\beta|_{\Gamma_0}$ is topologically free. (See \cite[Lemma 3.2]{Suz} for instance.)
Hence so is $\beta\times \sigma$.
In this case, the crossed product is a Kirchberg algebra in the UCT class,
since it is isomorphic to the tensor product of a matrix algebra and a Kirchberg algebra in the UCT class.
(See \cite[Section 1]{LaS} and \cite[Proposition 3.2]{Ana}.)
\end{Rem}

A particularly interesting case
is the one that the group $\Gamma$ is a free group.
In that case, the boundary space is the Cantor set.
\begin{Rem}\label{Rem:iso}
Let $2\leq n<\infty$.
We apply Theorem \ref{Thm:bou} to the case $\Gamma=\mathbb{F}_n$
and a transitive action $\sigma\colon \mathbb{F}_n\curvearrowright X$ on a set $X$
with $\sharp X=k\in \mathbb{N}$.
Then by Schreier's formula \cite[Chap.1, Prop.3.9]{LS}, any subgroup of $\mathbb{F}_n$ of the index $k$
is the free group of rank $m:=k(n-1)+1$.
Hence the resulting crossed product is isomorphic to $\mathbb{M}_k(C(\partial\mathbb{F}_m)\rtimes\mathbb{F}_m)$.
Thus the isomorphism class of it only depends on the cardinality of $X$.
However, the inclusion
$C(\partial\mathbb{F}_n)\rtimes \mathbb{F}_n\rightarrow C(\partial \mathbb{F}_n\times X)\rtimes \mathbb{F}_n\cong \mathbb{M}_k(C(\partial\mathbb{F}_m)\rtimes\mathbb{F}_m)$
does depend on the choice of transitive action.
The difference between them plays an important role in the next section.
\end{Rem}
\begin{Rem}\label{Rem:Cun}
For the case $\Gamma=\mathbb{F}_n$, a similar argument to that of Spielberg \cite{Spi} shows that
for a transitive action $\sigma\colon\mathbb{F}_n\curvearrowright \{1, \ldots, m\}$,
the crossed product of the diagonal action $\beta\times\sigma$ is isomorphic to the Cuntz--Krieger algebra
$\mathcal{O}_A$ with the matrix $A$ given as follows.
For $s\in S\sqcup S^{-1}$, set
\[A_s\delta_t:=\left\{ \begin{array}{ll}
\delta_s & {\rm if\ } t\neq s^{-1} \\
0 & {\rm otherwise}\\
\end{array} \right.\]
and denote by $B_s$ the permutation matrix of $\sigma(s)$.
Then the matrix $A$ is given by $\sum_{s\in\widetilde{S}}A_s\otimes B_s (\in\mathbb{M}_{\widetilde{S}\times X})$.
The elements of the crossed product corresponding to the canonical generators are given by
\[V_{s, x}:=(\chi_{\Omega(s)}\otimes\delta_{s.x})\lambda_s{\rm\ for\ } (s,x)\in \widetilde{S}\times X.\]
Since the $K_0$-groups of these crossed products are determined by \cite{Spi} and Remark \ref{Rem:iso},
we obtain a formula of elementary divisors for certain matrices.
We remark that a similar result still holds for (non-amenable) free products of finitely many cyclic groups.
The proof is again based on \cite{Spi}.
\end{Rem}

In the next section, we give a more powerful construction of
amenable minimal Cantor systems for (virtually) free groups.

\section{More general constructions of amenable minimal Cantor $\mathbb{F}_n$-systems}\label{sec:gmi}
In this and the next sections we investigate more general constructions of amenable minimal Cantor systems for free groups.
We construct continuously many amenable minimal Cantor systems
and classify them in terms of the crossed products.

For computations of $K$-groups in Theorem \ref{Thm:Inf}, we need a few lemmas and facts about the $K$-theory of the boundary algebras $C(\partial\mathbb{F}_n)\rtimes\mathbb{F}_n$.
J. Spielberg has shown that the boundary algebras of free groups are presented as a Cuntz--Krieger algebra \cite{Spi}.
(The canonical generators are explicitly given. See Remark \ref{Rem:Cun} or \cite{Spi}.)
This presentation and Cuntz's computation of the $K$-theory of Cuntz--Krieger algebras \cite[Proposition 3.1]{Cun}
show that the $K_0$-group of the boundary algebra is given by
$\left(\bigoplus_{s\in S}\mathbb{Z}[p_s]_0\right)\oplus \mathbb{Z}_{n-1}[1]_0$,
where $p_s$ denotes the characteristic function of the clopen subset $\Omega(s)$ for each $s\in \widetilde{S}$.
Here for an element $x$ of a group $G$,
we denote the subgroup $\langle x \rangle$ by $\mathbb{Z}x$ (resp. $\mathbb{Z}_m x$) if $x$ is of infinite order
(resp. $x$ is of order $m$).
Notice that for $s\in S$, the equality
$s.\Omega(s^{-1})=\partial\mathbb{F}_n\setminus \Omega(s)$ holds.
This implies
the equality $[p_s]_0+[p_{s^{-1}}]_0=[1]_0$.
We also have that the $K_1$-group is isomorphic to $\mathbb{Z}^n$.

We also need a few notations for abelian groups.
For an abelian group $G$, the torsion subgroup $G^{\rm tor}$ of $G$
is the subgroup of $G$ consisting of all torsion elements.
For a finitely generated abelian group $G$, denote by $G^{\rm free}$ the quotient group of $G/G^{\rm tor}$.
The subgroup $G^{\rm tor}$ is referred to as the torsion part of $G$
and the quotient $G^{\rm free}$ is referred to as the free part of $G$.
By the structure theorem of finitely generated abelian groups, $G^{\rm free}$ is indeed free abelian
and $G$ is isomorphic to $G^{\rm free}\oplus G^{\rm tor}$ (in a non-canonical way).
(Both of them are not true for more general groups.)
Every homomorphism $h$ between two finitely generated abelian groups
induces a homomorphism between their free parts.
We denote it by $h^{\rm free}$.
Similarly, for a homomorphism $h$ between two abelian groups,
we denote by $h^{\rm tor}$ the restriction of it to the torsion subgroup,
and refer to it as the torsion part of $h$.

Every automorphism $\varphi$ of $\mathbb{F}_n$ induces the automorphism $\Phi$
of the boundary algebra $C(\partial\mathbb{F}_n)\rtimes\mathbb{F}_n$ by
\[s\in \mathbb{F}_n\mapsto \varphi(s)\in \mathbb{F}_n\]
{\center and}
\[f\in C(\partial\mathbb{F}_n)\mapsto f\circ(\partial\varphi)^{-1}\in C(\partial\mathbb{F}_n).\]
Note that the mapping $\varphi\mapsto \Phi$ preserves the composition.
The next two lemmas are about the automorphism of $K_0(C(\partial\mathbb{F}_n)\rtimes\mathbb{F}_n)$
induced from $\Phi$.
\begin{Lem}\label{Lem:auto}
Let $2\leq n<\infty$ and fix an enumeration $S=\{ s_1, \ldots, s_n\}$ of $S$.
Then, for any $A\in\GL(n, \mathbb{Z})$, there is an automorphism $\varphi$ of $\mathbb{F}_n$
such that with respect to the following identification
\[K_0(C(\partial\mathbb{F}_n)\rtimes\mathbb{F}_n)=\left(\bigoplus_{i=1}^n\mathbb{Z}[p_{s_i}]_0\right)\oplus \mathbb{Z}_{n-1}[1]_0,\]
$K_0(\Phi)$ is identified with $A\oplus\id$.
\end{Lem}
\begin{proof}
We claim that for the following automorphisms
\begin{enumerate}[\upshape (i)]
\item the automorphism induced from a permutation of $S$,
\item the automorphism given by
\[s_1\mapsto s_2s_1 {\rm\ and\ } s_i\mapsto s_i {\rm\ for\ } 2\leq i\leq n,\]
\end{enumerate}
the equality $K_0(\Phi)={}^t({\varphi^{\rm ab})}^{-1}\oplus\id$ holds under the identification
of $\bigoplus_{s\in S}\mathbb{Z}[p_s]_0$ with the abelianization $\mathbb{F}_n^{\rm ab}$ of $\mathbb{F}_n$.
Then the same condition holds for any element of the subgroup of ${\rm Aut}(\mathbb{F}_n)$ generated
by these automorphisms.
This ends the proof.
The case (i) is obvious, so let $\psi$ be the automorphism of $\mathbb{F}_n$
given in the item (ii).
For our purpose, it is enough to expand $[\Psi(p_s)]_0$ as a linear combination of $[p_t]_0$'s.
By the definition of $\Psi$, for each $s\in S$, the projection $\Psi(p_s)$ is the characteristic function of
$\partial\psi(\Omega(s; S))=\Omega(\psi(s); \psi(S))$.
Then one can check easily that the following three equations hold.
\[\Omega(\psi(s_1); \psi(S))=s_2.\Omega(s_1; S).\]
\[\Omega(\psi(s_2); \psi(S))=\Omega(s_2; S)\setminus(s_2.\Omega(s_1; S)).\]
\[\Omega(\psi(s_i); \psi(S))=\Omega(s_i; S) {\rm\ for\ }2<i\leq n.\]
Here we only give a proof of the inclusion $s_2.\Omega(s_1; S)\subset \Omega(\psi(s_1); \psi(S))$.
The rest of the proof is done in a similar way.
Let $w\in s_2.\Omega(s_1, S)$ be given.
Then the reduced form of $w$ is of the form
$s_2s_1v_1s_1^{k_1}s_2^{l_1}v_2\cdots$
for some $v_i \in \langle s_3, \ldots, s_n \rangle$, $k_i, l_i \in \mathbb{Z}$.
(Here we allow the possibility that $v_i=e$, $k_i=0$, $l_i=0$.)
Then the expansion of $w$ with respect to $\psi(S)$
is given by reducing the formal infinite product
$\psi(s_1)v_1 (\psi(s_2)^{-1}\psi(s_1))^{k_1}\psi(s_2)^{l_1} v_2\cdots$ to a reduced form (with respect to $\psi(S)$).
Note that since $w=s_2s_1v_1s_1^{k_1}s_2^{l_1}v_2\cdots$ is reduced with respect to $S$,
the equality $v_1=e$ implies $k_1\geq 0$ and similarly for the other places.
This shows that any cancellation does not remove the first $\psi(s_1)$.
Thus we have $s_2.\Omega(s_1, S) \subset \Omega(\psi(s_1), \psi(S))$.

From these equations, we obtain the equations
\[\Psi([p_{s_1}]_0)=[p_{s_1}]_0,\]
\[\Psi([p_{s_2}]_0)=[p_{s_2}]_0-[p_{s_1}]_0,\]
\[\Psi([p_{s_i}]_0)=[p_{s_i}]_0{\rm\ for\ }2<i\leq n.\]
This shows that the automorphism $\psi$ satisfies our claim.
\end{proof}

\begin{Lem}\label{Lem:auto2}
Let $t, u$ be two distinct element of $S$ and let $m\in\mathbb{Z}$.
Let $\psi$ be the automorphism of $\mathbb{F}_n$ defined by
\[t\mapsto t, u\mapsto t^mu t^{-m},{\rm\ and\ }v\mapsto v{\rm\ for\ the\ other\ } v\in S.\]
Then $K_0(\Psi)$ is given by
\[[1]_0\mapsto [1]_0, [p_t]_0\mapsto [p_t]_0-m[1]_0,{\rm\ and\ }[p_v]_0\mapsto [p_v]_0 {\rm\ for\ }u\in S\setminus\{t\}.\]
\end{Lem}
\begin{proof}
Let $\varphi_1, \varphi_2 $ be the automorphisms of $\mathbb{F}_n$ defined by
$\varphi_1(u)= u^{-1}, \varphi_1(v) := v {\rm\ for\ }v\in S\setminus\{u\},$
$\varphi_2(u) := tu,$ and $\varphi_2(v) := v{\rm \ for\ }v\in S\setminus\{u\}$.
Then a direct computation shows the equality
$\psi=\varphi_2\circ \varphi_1\circ\varphi_2\circ \varphi_1$.
Therefore, to compute $K_0(\Psi)$, it suffices to compute $K_0(\Phi_1)$ and $K_0(\Phi_2)$.
The computation of $K_0(\Phi_1)$ is easily derived from the equation $[p_{u^{-1}}]_0=[1]_0-[p_u]_0$.
The $K_0(\Phi_2)$ is computed in the proof of Lemma \ref{Lem:auto}.
Now the claim follows from a simple algebraic computation.
\end{proof}
\begin{Lem}\label{Lem:Conn}
Let $\mathbb{F}_n$ be the free group.
Enumerate $S$ as $S:=\{s_1,\ldots, s_n\}$.
Consider the presentation
\[K_0(C(\partial\mathbb{F}_n)\rtimes\mathbb{F}_n)=\left(\bigoplus_{i=1}^{n}\mathbb{Z}[p_{s_i}]_0\right)\oplus \mathbb{Z}_{n-1}[1]_0,\]
where $p_s:=\chi_{\Omega(s; S)}$ for $s\in S\sqcup S^{-1}$.
Similarly, for each enumerated free basis $W:=\{w_1, \ldots, w_m\}$ of a finite index subgroup $\Gamma$ of $\mathbb{F}_n$,
consider the presentation
\[K_0(C(\partial\mathbb{F}_n\times(\mathbb{F}_n/\Gamma))\rtimes\mathbb{F}_n)=\left(\bigoplus_{i=1}^m\mathbb{Z}[q_{(w_i; W)}]_0\right)\oplus \mathbb{Z}_{m-1}[r_W]_0\]
where $q_{(w; W)}:=\chi_{\Theta(w;W)}\otimes\delta_{[e]}$ for $w\in W\sqcup W^{-1}$
and $r_W:=1\otimes\delta_{[e]}$.
$($This follows from the isomorphism given in Theorem $\ref{Thm:bou}$.
See also the proof below.$)$
Let
$$j \colon C(\partial\mathbb{F}_n)\rtimes\mathbb{F}_n\rightarrow C(\partial\mathbb{F}_n\times(\mathbb{F}_n/\Gamma))\rtimes\mathbb{F}_n$$
denote the canonical inclusion.
Let $A\colon \mathbb{Z}^n\rightarrow\mathbb{Z}^{n}$ be an injective homomorphism
and let $l$ denote the product of all elementary divisors of $A$.
Then for any left invertible inclusion
$Q\colon\mathbb{Z}^n\rightarrow\mathbb{Z}^{l(n-1)+1}$,
there exists an enumerated finite subset $W=\{w_1,\ldots, w_m\}$ of $\mathbb{F}_n$
satisfying the following conditions.
\begin{itemize}
\item The subset $W$ is a free basis of a subgroup $\Gamma$ of $\mathbb{F}_n$.
\item The index $[\mathbb{F}_n:\Gamma]$ is $l$ $($hence $m=l(n-1)+1)$.
\item With respect to the above enumerated bases,
$K_0(j)^{\rm free}$ is presented by $QA$.
\item The torsion part of $K_0(j)$ is injective.
\item The image of $\bigoplus_{i=1}^{n}\mathbb{Z}[p_{s_i}]_0$ under $K_0(j)$ is contained in the subgroup
\[\left(\bigoplus_{i=1}^m\mathbb{Z}[q_{(w_i; W)}]_0\right)\oplus \Lambda\]
where $\Lambda$ denotes the subgroup of $\mathbb{Z}_{m-1}[r_W]_0$ generated by
elements of order $2$, which must be either trivial or isomorphic to $\mathbb{Z}_2$.
\end{itemize}
\end{Lem}
\begin{proof}
First we show that if the claim holds for a homomorphism $QA\colon \mathbb{Z}^n\rightarrow\mathbb{Z}^{l(n-1)+1}$,
then it also holds for any homomorphisms of the form $BQAC$
where $B\in\GL({l(n-1)+1}, \mathbb{Z})$ and $C\in\GL(n, \mathbb{Z})$.
To see this, take an enumerated finite subset $W$ which satisfies the required conditions for $QA$.
Clearly, it suffices to show it for the case either $B=\id$ or $C=\id$ holds.

First we deal the case $B=\id$.
In this case, take an automorphism $\varphi$ of $\mathbb{F}_{n}$ satisfying
$K_0(\Phi)=C^{-1}\oplus\id$, which exists by Lemma \ref{Lem:auto}.
Consider the commutative diagram
$$
\begin{CD}
C(\partial\mathbb{F}_n)\rtimes\mathbb{F}_n@> j >>C(\partial\mathbb{F}_n\times(\mathbb{F}_n/\Gamma))\rtimes\mathbb{F}_n\\
@V \Phi VV @V \widetilde{\Phi} VV\\
C(\partial\mathbb{F}_n)\rtimes\mathbb{F}_n@>j' >>C(\partial\mathbb{F}_n\times(\mathbb{F}_n/\varphi(\Gamma)))\rtimes\mathbb{F}_n
\end{CD}
$$
where each row map is the canonical inclusion
and the second column map $\widetilde{\Phi}$ is the isomorphism
induced from the following covariant representation
\[s \in \mathbb{F}_n\mapsto \varphi(s)\in \mathbb{F}_n,\]
\[f\in C(\partial\mathbb{F}_n \times (\mathbb{F}_n/\Gamma))\mapsto (f\circ(\partial\varphi\times \tilde{\varphi})^{-1})\in C(\partial\mathbb{F}_n\times(\mathbb{F}_n/\varphi(\Gamma)))\rtimes\mathbb{F}_n.\]
Here $\tilde{\varphi}\colon \mathbb{F}_n/\Gamma\rightarrow \mathbb{F}_n/\varphi(\Gamma)$
is the bijection defined by $x\Gamma\in \mathbb{F}_n/\Gamma\mapsto \varphi(x\Gamma)\in \mathbb{F}_n/\varphi(\Gamma)$
Then on the level of $K_0$-groups, the above commutative diagram becomes
the following commutative diagram.
$$
\begin{CD}
\mathbb{Z}^S\oplus\mathbb{Z}_{n-1}@>K_0(j) >>\mathbb{Z}^W\oplus\mathbb{Z}_{m-1}\\
@V C^{-1}\oplus{\rm id} VV @V\tau_\varphi VV\\
\mathbb{Z}^S\oplus\mathbb{Z}_{n-1}@> K_0(j')>>\mathbb{Z}^{\varphi(W)}\oplus\mathbb{Z}_{m-1}
\end{CD}
$$
where $\tau_\varphi$ is the isomorphism given by $[q_{(w: W)}]_0\mapsto[q_{(\varphi(w);\varphi(W))}]_0$ and $[r_W]_0\mapsto [r_{\varphi(W)}]_0$.
This shows that
the enumerated finite subset $\varphi(W)$ satisfies the desired conditions for the homomorphism $QAC$.
We remark that this operation may change the subgroup $\Gamma$.

Next we consider the case $C=\id$.
In this case, take an automorphism $\psi$ of $\Gamma$ such that
the induced automorphism $\Psi$ that satisfies
$K_0(\Psi)=B^{-1}\oplus\id$ with respect to the enumerated free basis $W$.
Then from this form of $K_0(\Psi)$,
we immediately conclude that the enumerated finite subset $\psi(W)$ satisfies the desired conditions for
$BQA$. 
(The point here is that all operations do not change the enumerated free basis $S$.)

From this together with the elementary divisor theory, it suffices to show the assertion for the case $A$ is a diagonal homomorphism
with respect to the standard basis.
By decomposing $A$ as a composite of finitely many homomorphisms, we only need to show the following.
For each $k \in \mathbb{N}\setminus\{1\}$,
there is a free basis $W$ of a finite index subgroup $\Gamma$ of $\mathbb{F}_n$
with the following properties.
\begin{itemize}
\item The index $[\mathbb{F}_n: \Gamma]$ is $k$.
\item The elementary divisors of $K_0(j)$
are given by $(1, 1, \ldots, 1, k)$.
\item The $K_0(j)$ satisfies the last two conditions in the statement.
\end{itemize}
To construct the desired $W$, fix $s\in S$ and
set \[W:=\left\{s^k, s^lts^{-l}:0\leq l\leq k-1, t\in S\setminus\{s\}\right\}.\]
Then $W$ is a free basis of the kernel $\Gamma$ of the homomorphism
$\pi_s\colon\mathbb{F}_n\rightarrow\mathbb{Z}_k$.
Here $\pi_s$ is given by
\[s\mapsto [1]{\rm\ and\ }t\mapsto [0]{\rm\ for\ }t\in S\setminus\{s\}.\]
In particular, the subgroup of $\mathbb{F}_n$ generated by $W$ is of index $k$.
Note that the group $\Gamma$ coincides with the stabilizer subgroup of the action
$\sigma\colon\mathbb{F}_n\curvearrowright \mathbb{F}_n/\Gamma$ of an arbitrary point.

By direct computations, we obtain the equalities
\[\Omega(s)=\bigsqcup_{w\in I} \Theta(w; W)\]
where $I:=\{ s^k,s^{l}t^{\pm 1}s^{-l}: 1\leq l\leq k-1, t\in S\setminus\{s\}\},$ and
\[\Omega(t)=\Theta(t; W){\rm\ for\ } t\in S\setminus\{s\}.\]

From the isomorphism
$C(\partial\mathbb{F}_n\times(\mathbb{F}_n/\Gamma))\rtimes\mathbb{F}_n\cong\mathbb{M}_{k}(C(\partial\Gamma)\rtimes\Gamma)$
given in Theorem \ref{Thm:bou},
the canonical (non-unital) inclusion
\[C(\partial\mathbb{F}_n\times\{[e]\})\rtimes\Gamma\rightarrow C(\partial\mathbb{F}_n\times(\mathbb{F}_n/\Gamma))\rtimes\mathbb{F}_n\]
induces the isomorphism of $K_0$-groups.
This yields the equation \[K_0(C(\partial\mathbb{F}_n\times(\mathbb{F}_n/\Gamma))\rtimes\mathbb{F}_n)=\left( \bigoplus_{w\in W} \mathbb{Z}[q_{(w; W)}]_0 \right) \oplus \mathbb{Z}_{k(n-1)}[r_W]_0.\]

Since the set ${\rm im}(\rho)=\{s^l:0\leq l\leq k-1\}$ is a complete system of representatives for the quotient
$\mathbb{F}_n/\Gamma$, the above equations of $\Theta$'s give the equations
\[p_s=\sum_{l=0}^{k-1}\sum_{v\in I}q_{(v; W)}\circ(\beta(s^{-l})\times\id)\circ({\gamma(s^l)})\]
and
\[p_t=\sum_{l=0}^{k-1}q_{(t; W)}\circ(\beta(s^{-l})\times\id)\circ({\gamma(s^{l})})\]
for $t\in S\setminus\{s\}$.

(Here we use the equality $(\beta(s^{-l})\times\id)\circ\gamma(s^l)=\id\times\sigma(s^l)$.)
We also have the equaltion
\[1=\sum_{l=0}^{k-1}r_W\circ(\gamma(s^{l})).\]
The last equation shows
the equation $[1]_0=k[r_W]_0$.
This shows that the $K_0(j)$ preserves
the order of the unit $[1]_0$.
This proves the injectivity of $K_0(j)^{\rm tor}$.

Observe that the homeomorphism $\beta(s^{-1})$ on $\partial\Gamma$
is induced by a group automorphism of $\Gamma$.
Indeed, it is induced from the conjugating automorphism $\alpha:=\ad(s^{-1})$.
(We remark that this $\alpha$ is not inner.)
Hence it extends to the automorphism $\Phi$ of
$C(\partial \Gamma)\rtimes \Gamma\cong C(\partial\mathbb{F}_n\times\{[e]\})\rtimes\Gamma$.
The proof of Lemma \ref{Lem:auto} and Lemma \ref{Lem:auto2} show that $K_0(\Phi)$ is the automorphism of
$\left( \bigoplus_{w\in W}\mathbb{Z}[q_{(w; W)}]_0 \right) \oplus\mathbb{Z}_{m-1}[r_W]_0$ given as follows.
\[[r_W]_0\mapsto [r_W]_0,\]
\[ [q_{(s^k; W)}]_0\mapsto[q_{(s^k; W)}]_0+(n-1)[r_W]_0\]
\[[q_{(w; W)}]_0\mapsto [q_{(\sigma(w); W)}]_0{\rm\ for\ }w\in W\setminus\{s^k\},\]
where $\sigma$ is the permutation of $W$ given by
\[\sigma(w):=\left\{ \begin{array}{ll}
w & {\rm if\ } w=s^k \\
s^{-1+k}ws^{1-k} & {\rm if\ } w\in S\setminus\{s\}\\
s^{-1}ws & {\rm\ otherwise}\\
\end{array}.\right.\]
This is because the automorphism $\alpha$ is equal to
the composite of the automorphism induced from the permutation $\sigma$ of $W$ and
$n-1$ automorphisms of the form appearing in Lemma \ref{Lem:auto2} with $t=s^k$ and $m=-1$.
(Notice that the equality $s^{-1}ts=s^{-k}(s^{k-1}ts^{1-k})s^{k}$ holds for $t\in S$.)

From this, the first equation is reduced to
\begin{eqnarray*}[p_s]_0&=&
\sum_{l=0}^{k-1}\left([q_{(s^k; W)}]_0+l(n-1)[r_W]_0+(k-1)(n-1)[r_W]_0\right)\\
&=&k[q_{(s^k; W)}]_0+\frac{k(k-1)(n-1)}{2}[r_W]_0.
\end{eqnarray*}
Here we use the equation $[q_{(w; W)}]_0+[q_{(w^{-1}; W)}]_0=[r_W]_0$ for $w\in W$.
Similarly, the second equation is reduced to
\[[p_t]_0=\sum_{v\in J_t}[q_{(v; W)}]_0,\]
where $J_t=\{s^{l}ts^{-l}:0\leq l\leq k-1\}$, for $t\in S\setminus \{s\}$.
Since the sets $\{s^k\}, J_t; t\in S\setminus\{s\}$ are mutually disjoint
subsets of $W$
and the order of the element $\frac{k(k-1)(n-1)}{2}[r_W]_0$
is either $0$ or $2$, 
this $W$ is what we needed.
\end{proof}
For a torsion group $H$, we associate a supernatural number $N_H$ by
\[N_H:={\rm lcm}\{\sharp K:K{\rm\ is\ a\ finite\ subgroup\ of\ }H\}.\]
For a subset $X$ of $\mathbb{Z}^{\oplus\infty}$, define the subgroup $P(X)$ of $\mathbb{Z}^{\oplus\infty}$ to be
\[\{y\in \mathbb{Z}^{\oplus\infty} :{\rm there\ exists\ } n\in\mathbb{Z}\setminus\{0\} {\rm \ with\ } ny\in \langle X\rangle\}.\]
We denote by $\pi$ the quotient homomorphism $\mathbb{Q}^{\oplus\infty}\rightarrow\mathbb{Q}^{\oplus\infty}/\mathbb{Z}^{\oplus\infty}$.
\begin{Lem}\label{Lem:Ind}
Let $G$ be a subgroup of $\mathbb{Q}^{\oplus \infty}$
which contains $\mathbb{Z}^{\oplus\infty}$ as a subgroup of infinite index.
Let $2\leq n<\infty$.
Then $G$ is isomorphic to the inductive limit of an inductive system of the form
$(\mathbb{Z}^{k_m}, A_m)_m$ that satisfies the following conditions.
\begin{itemize}
\item Each connecting map $A_m$ is injective.
\item The sequence $(k_m)_m$ satisfies $k_1=n$ and $k_m=l_{m-1}(k_{m-1}-1)+1$ for $m\geq 2$,
where for each $m$, $l_m$ denotes the product of all elementary divisors of $A_m$.
\item The formal infinite product $\prod_m l_m$ is equal to $N_{G/\mathbb{Z}^{\oplus\infty}}$.
\end{itemize}
\end{Lem}
\begin{proof}
Fix an enumeration $\{x_n\}_n$ of elements of $\mathbb{Z}^{\oplus\infty}$.
Since the quotient $G/ \mathbb{Z}^{\oplus \infty}$ is an infinite torsion abelian group,
there are a sequence $(y_j)_j$ of elements of
$\mathbb{Z}^{\oplus\infty}$ and a sequence $(k_j)_j$ of natural numbers greater than $1$
such that the sequence $(\langle\pi (k_1^{-1}y_1), \ldots, \pi (k_j^{-1}y_j)\rangle)_j$ of subgroups of $\pi(G)$ is strictly increasing and the union of the sequence coincides with $\pi(G)$.
Note that the set $\{x_j, k_j^{-1}y_j: j\in\mathbb{N}\}$ generates $G$.
For each $j\in \mathbb{N}$, set $m_j:=\sharp(\langle\pi (k_1^{-1}y_1), \ldots, \pi (k_j^{-1}y_j)\rangle).$
Take $r_1\in\mathbb{N}$
such that the rank of $\langle y_1, x_1, \ldots, x_{r_1}\rangle$ is equal to $n$
and set $H_1:=P(y_1, x_1, \ldots, x_{r_1})$.
Note that since any subgroup of $\mathbb{Z}^{\oplus\infty}$ is free abelian \cite[Vol.I,\ Theorem 14.5]{Fuc},
a subgroup of $\mathbb{Z}^{\oplus\infty}$ of rank $k$ is isomorphic to $\mathbb{Z}^k$
for any $0\leq k \leq \infty$.
In particular $H_1$ is isomorphic to $\mathbb{Z}^n$.
Next take $r_2\in\mathbb{N}$ such that $r_2\geq r_1$ and
the rank of 
$\langle y_1, y_2, x_1, \ldots, x_{r_2}\rangle$ is equal to $m_1(n-1)+1$.
Set the subgroup $H_2$ of $G$ to be
$\langle k_1^{-1}y_1, P(y_1, y_2, x_1, \ldots, x_{r_2})\rangle$.
Then $H_2$ contains $H_1$, the rank of $H_2$ is $m_1(n-1)+1$, and $H_2$ is finitely generated. Hence $H_2$ is isomorphic to $\mathbb{Z}^{m_1(n-1)+1}$.
We will determine the product $l_1$ of all elementary divisors of the inclusion map $\iota_1\colon H_1\rightarrow H_2$.
This is equal to the order of $(H_2/H_1)^{\rm tor}$.
By definition of $H_1$ and $H_2$, the group $(H_2/H_1)^{\rm tor}$
is generated by the image of $k_1^{-1}y_1$.
Since $\langle k_1^{-1}y_1\rangle \cap H_1=\langle k_1^{-1}y_1\rangle \cap \mathbb{Z}^{\oplus\infty}$,
the group $(H_2/H_1)^{\rm tor}$ is isomorphic to $\langle \pi(k_1^{-1}y_1) \rangle$.
Hence we have $l_1=m_1$.
Next take $r_3\in\mathbb{N}$ such that
the rank of the group $\langle y_1, y_2, y_3, x_1, \ldots, x_{r_3}\rangle$ is equal to $m_2(n-1)+1$
and set $H_3:=\langle k_1^{-1}y_1, k_2^{-1}y_2, P(y_1, y_2, y_3, x_1, \ldots, x_{r_3}) \rangle$.
Note that since $m_2(n-1)-m_1(n-1)=m_1(l_1-1)(n-1)>1$, we must have $r_3> r_2$.
It is clear from the definition that $H_3$ contains $H_2$.
By a similar reason to above, the group $H_3$ is isomorphic to $\mathbb{Z}^{m_2(n-1)+1}$.
We determine the product $l_2$ of all elementary divisors of the inclusion map $\iota_2 \colon H_2\rightarrow H_3$.
By the same reason as above, it is equal to the order of $(H_3/H_2)^{\rm tor}$.
It is clear that $(H_3/H_2)^{\rm tor}$ is generated by the image of $k_2^{-1}y_2$.
This with the equality $\langle k_2^{-1}y_2\rangle\cap H_2 =\langle k_2^{-1}y_2\rangle \cap \langle k_1^{-1}y_1, \mathbb{Z}^{\oplus\infty}\rangle$ shows that
the group $(H_3/H_2)^{\rm tor}$ is isomorphic to $\langle\pi (k_1^{-1}y_1), \pi(k_2^{-1}y_2)\rangle/\langle \pi(k_1^{-1}y_1)\rangle$.
This shows the equation $l_2=m_2/m_1$.
Continuing this process inductively, we obtain an increasing sequence
$(H_j)_j$ of subgroups of $G$ which has the following properties.
\begin{itemize}
\item The union of $H_j$'s is equal to $G$.
\item The group $H_j$ is isomorphic to $\mathbb{Z}^{m_{j-1}(n-1)+1}$ for each $j\in \mathbb{N}$.
Here we put $m_0=1$ for convenience.
\item The product $l_j$ of all elementary divisors of the inclusion map
$\iota_j \colon H_j \rightarrow H_{j+1}$ is equal to $m_j/m_{j-1}$ for each $j\in \mathbb{N}$.
\end{itemize}
By the last property, we have
$\prod_j l_j={\rm lcm}\{m_j : j \in \mathbb{N}\}$,
which is equal to $N_{G/\mathbb{Z}^{\oplus\infty}}$.
Therefore the inductive system $(H_m, \iota_m)_m$ satisfies the desired properties.
\end{proof}
Now we prove the main theorem.
\begin{Thm}\label{Thm:Inf}
Let $G$ be a subgroup of $\mathbb{Q}^{\oplus\infty}$ which contains $\mathbb{Z}^{\oplus\infty}$ as a subgroup
of infinite index.
Set $\tilde{G}:=G/\mathbb{Z}^{\oplus\infty}$.
Let $2\leq n<\infty$ and $k$ be an integer.
Then there exists an amenable minimal Cantor $\mathbb{F}_n$-system that satisfies the following properties.
\begin{itemize}
\item The pair of $K_0$-group and the unit $[1]_0$ is isomorphic to 
\[\left(G\oplus\Lambda((n-1)N_{\tilde{G}}), 0\oplus [k(n-1)^{-1}]\right).\]
\item The $K_1$-group is isomorphic to $\mathbb{Z}^{\oplus \infty}$.
\item The crossed product is a Kirchberg algebra in the UCT class.
\end{itemize}
\end{Thm}
\begin{proof}
First we consider the case $k=1$.
Let $G$ be a group as stated.
Take an inductive system $(\mathbb{Z}^{k_m}, A_m)_m$ as in Lemma \ref{Lem:Ind}
for the case of $n$ and $G$.
Fix an enumeration of $S$.
By Lemma \ref{Lem:Conn}, there is an enumerated subset $W_1$ of $\mathbb{F}_n$
such that the equation $\sharp W_1=(n-1)l_1+1=k_2$ holds,
the subset $W_1$ is a free basis of a finite index subgroup $\Gamma_1$ of $\mathbb{F}_n$,
and the $K_0$-map 
$h_1:=K_0(j_1)$
of the canonical inclusion
\[j_1\colon C(\partial\mathbb{F}_n)\rtimes\mathbb{F}_n\rightarrow C(\partial\mathbb{F}_n\times(\mathbb{F}_n/\Gamma_1))\rtimes\mathbb{F}_n\]
satisfies the following conditions.
\begin{itemize}
\item With respect to the bases $([\chi_{\Omega(s)}]_0:s\in S)$ and $([\chi_{\Theta(w_1; W_1;\mathbb{F}_n)}\otimes\delta_e]_0: w_1\in W_1)$,
the free part of $h_1$ is presented by $A_1$.
\item The image of $\bigoplus_{s\in S}\mathbb{Z}[\chi_{\Omega(s)}]_0$ under $h_1$
is contained in
\[\left( \bigoplus_{w_1\in W_1} \mathbb{Z}[\chi_{\Theta(w_1; W_1;\mathbb{F}_n)}]_0 \right)\oplus \Lambda_1\]
where $\Lambda_1$ is the subgroup of the torsion part generated by elements of order $2$.
\item The torsion part of $h_1$ is injective.
\end{itemize}
From Lemma \ref{Lem:Conn}
and the proof of Theorem \ref{Thm:bou},
we can further take an enumerated subset $W_2$ of $\Gamma_1$ such that
the equation $\sharp W_2=(\sharp W_1-1)l_2+1=k_3$ holds,
the subset $W_2$ is a free basis of a finite index subgroup $\Gamma_2$ of $\Gamma_1$,
and the $K_0$-map $h_2:=K_0(j_2)$ of the canonical inclusion
\[j_2\colon C(\partial\mathbb{F}_n\times(\mathbb{F}_n/\Gamma_1))\rtimes\mathbb{F}_n\rightarrow C(\partial\mathbb{F}_n\times(\mathbb{F}_n/\Gamma_2))\rtimes\mathbb{F}_n\]
satisfies the following conditions.
\begin{itemize}
\item With respect to the bases $([\chi_{\Theta(w_1; W_1;\mathbb{F}_n)}\otimes\delta_e]_0: w_1\in W_1)$
and $([\chi_{\Theta(w_2; W_2;\mathbb{F}_n)}\otimes\delta_e]_0: w_2\in W_2)$,
the free part of $h_2$ is presented by $A_2$.
\item The image of $\bigoplus_{w_1\in W_1}\mathbb{Z}[\chi_{\Theta(w_1; W_1;\mathbb{F}_n)}\otimes\delta_e]_0$ under $h_2$ is contained in
\[\left( \bigoplus_{w_2\in W_2} \mathbb{Z}[\chi_{\Theta(w_2; W_2;\mathbb{F}_n)}\otimes\delta_e]_0 \right) \oplus \Lambda_2,\]
where $\Lambda_2$ is the subgroup of the torsion part generated by elements of order $2$.
\item The torsion part of $h_2$ is injective.
\end{itemize}
Note that we have $h_1(\Lambda_1)\subset\Lambda_2$ and each $\Lambda_i$
is either trivial or isomorphic to $\mathbb{Z}_2$.

Continuing this process inductively, we finally obtain a sequence $(W_m)_m$ of enumerated subsets of $\mathbb{F}_n$
such that each $W_m$ is a free basis of a finite index subgroup of $\Gamma_{m-1}:=\langle W_{m-1}\rangle$
(here we set $W_0:=S$ for convenience),
the equation $\sharp W_m=k_{m+1}$ holds,
and for each $m$, the $K_0$-map
$h_m:=K_0(j_m)$
of the canonical inclusion
\[j_m\colon C(\partial\mathbb{F}_n\times(\mathbb{F}_n/\Gamma_{m-1}))\rtimes\mathbb{F}_n\rightarrow C(\partial\mathbb{F}_n\times(\mathbb{F}_n/\Gamma_m))\rtimes\mathbb{F}_n\]
satisfies the analogues of above three conditions.

For each $m$ we denote by $\alpha_m$ the canonical action of $\mathbb{F}_n$ on $Y_m:=\mathbb{F}_n/\Gamma_m$.
Denote by $\alpha\colon \mathbb{F}_n\curvearrowright Y$ the projective limit of the projective system $(\alpha_m\colon \mathbb{F}_n\curvearrowright Y_m)_m$.
Set $X_m:=\partial\mathbb{F}_n\times Y_m$, $X:=\partial\mathbb{F}_n\times Y$, $\gamma_m:=\beta\times\alpha_m$, and $\gamma:=\beta\times\alpha$.
Note that by definition $\gamma$ is the projective limit of $(\gamma_m)_m$.
Then since both amenability and minimality passes to projective limits,
the $\gamma$ is an amenable minimal Cantor system.
Since the UCT class and the class of all Kirchberg algebras are closed under inductive limits,
the crossed product of $\gamma$ is a Kirchberg algebra in the UCT class.
We next show that the Cantor $\mathbb{F}_n$-system $\gamma$ has the desired $K$-theory.

First we determine the pair $(K_0(C(X)\rtimes\mathbb{F}_n), [1]_0)$.
From our construction, we have an isomorphism
\[K_0(C(X)\rtimes\mathbb{F}_n)\cong\varinjlim(\mathbb{Z}^{k_m}\oplus \mathbb{Z}_{k_m-1}, h_m).\]
This shows that the group $K_0(C(X)\rtimes\mathbb{F}_n)/K_0(C(X)\rtimes\mathbb{F}_n)^{\rm tor}$ is isomorphic to $G$ and
the group $K_0(C(X)\rtimes\mathbb{F}_n)^{\rm tor}$ is isomorphic to $\Lambda((n-1)N_{\tilde{G}})$.
(Note that $k_{m+1}-1=(n-1)l_1\cdots l_m$ for $m\in \mathbb{N}$.)
We show that the torsion subgroup $K_0(C(X)\rtimes\mathbb{F}_n)^{\rm tor}$ is a direct summand of $K_0(C(X)\rtimes\mathbb{F}_n)$.
This proves that the group $K_0(C(X)\rtimes\mathbb{F}_n)$ is isomorphic to the expected group.
To see this, consider the subgroup $H$ of $K_0(C(X)\rtimes\mathbb{F}_n)$
generated by elements of order $2$, which exists at most one,
and the images of $[\chi_{\Theta(w_m; W_m;\mathbb{F}_n)}\otimes\delta_e]_0$
for all $m\in \mathbb{N}$ and $w_m\in W_m$.
Then the torsion part of $H$ is either trivial or of order $2$.
In both cases, Szele's Theorem \cite[Vol.\ I.\ Prop.\ 27.1]{Fuc} shows that the $H^{\rm tor}$ is a direct summand of $H$.
This shows that the torsion free quotient $K_0(C(X)\rtimes\mathbb{F}_n)/K_0(C(X)\rtimes\mathbb{F}_n)^{\rm tor}$
is lifted to $H\subset K_0(C(X)\rtimes\mathbb{F}_n)$ by homomorphism, as desired.
Furthermore, by the construction, the above isomorphism maps the unit $[1]_0$
to $0\oplus[(n-1)^{-1}]$.
This shows that the $\gamma$ has the desired $K_0$-group.

Next we determine the $K_1$-group.
By the Pimsner--Voiculescu exact sequence for free groups \cite{PV},
for any Cantor $\mathbb{F}_n$-system $\tau\colon \mathbb{F}_n\curvearrowright Z$, we obtain the isomorphism
$K_1(C(Z)\rtimes\mathbb{F}_n)\cong\ker(\eta_{\tau})$,
where $\eta_{\tau}\colon C(Z, \mathbb{Z})^{\oplus S}\rightarrow C(Z, \mathbb{Z})$ is the group homomorphism given by
\[(f_s)_{s \in S} \in C(Z, \mathbb{Z})^{\oplus S} \mapsto \sum_{s\in S}(f_s-f_s\circ\tau(s^{-1})) \in C(Z, \mathbb{Z}).\]

From the above isomorphism and the functoriality of the Pimsner--Voiculescu exact sequence,
the canonical map $K_1(C(X_m)\rtimes\mathbb{F}_n)\rightarrow K_1(C(X)\rtimes\mathbb{F}_n)$
is injective for each $m$ (since $\eta_{\gamma_{m}}$ is identified with the restriction of $\eta_{\gamma}$).
The isomorphism
\[C(X_m)\rtimes\mathbb{F}_n\cong\mathbb{M}_{L_m}(C(\partial\mathbb{F}_{k_m})\rtimes\mathbb{F}_{k_m}),\]
shows that the rank of $K_1(C(X_m)\rtimes\mathbb{F}_n)$ is $k_m$.
Here $L_m:=l_1\cdots l_m$.
This shows that the rank of $K_1(C(X)\rtimes\mathbb{F}_n)$ must be infinite.
Since the group $K_1(C(X)\rtimes\mathbb{F}_n)$ is a subgroup of the free abelian group $C(X, \mathbb{Z})^{\oplus S}$, it is free abelian \cite[Vol.I,\ Theorem 14.5]{Fuc}.
This shows that the $K_1$-group is isomorphic to $\mathbb{Z}^{\oplus\infty}$. 

To end the proof for general case, we need the skyscraper construction.
For $G$ as above, let $\gamma_G\colon \mathbb{F}_n \curvearrowright X$ be the Cantor $\mathbb{F}_n$-system
constructed in the above for the case of $G$.
Then for each natural number $k$, define the new Cantor $\mathbb{F}_n$-system $\widetilde{\gamma_G}^{(k)}$ as follows.
Set
$\widetilde{X}^{(k)}:=X\times\{1, \ldots, k\}$
and fix $s\in S$.
Then define a dynamical system $\widetilde{\gamma_G}^{(k)}$ of $\mathbb{F}_n$ on $\widetilde{X}^{(k)}$ by
\[\widetilde{\gamma_G}^{(k)}(s)(x, j):=\left\{ \begin{array}{ll}
(x, j+1) & {\rm if\ } j\neq k \\
(\gamma_G(s)(x), 1) & {\rm if\ } j=k\\
\end{array} \right.\]
and $\widetilde{\gamma_G}^{(k)}(t)(x, j):=(\gamma_G(t)(x), j)$ for the other $t\in S$.
Then by definition, $\widetilde{\gamma_G}^{(k)}$ is an amenable minimal Cantor system
and its crossed product is isomorphic to the tensor product of $\mathbb{M}_k$ and the crossed product of $\gamma_G$.
This ends the proof.
\end{proof}

\begin{Cor}\label{Cor:Inf}
Every free group admits continuously many amenable minimal Cantor systems
which are distinguished by the $K$-groups.
\end{Cor}

Similar results also hold for virtually free groups.
\begin{Thm}\label{Thm:Infv}
Let $\Gamma$ be a virtually $\mathbb{F}_n$ group.
Let $k$ be a natural number such that $\Gamma$ has
a subgroup $\Lambda$ of index $k$ isomorphic to $\mathbb{F}_n$.
Then for each triple $(G_0, u, G_1)$ given in Theorem $\ref{Thm:Inf}$ for the case $n$,
there exists an amenable minimal topologically free Cantor $\Gamma$-system that satisfies the following properties.
\begin{itemize}
\item The $K$-theory $(K_0, [1]_0, K_1)$ is isomorphic to $(G_0, ku, G_1)$.
\item The crossed product is a Kirchberg algebra in the UCT class.
\end{itemize}
In particular, the same conclusion in Corollary $\ref{Cor:Inf}$ also holds for virtually $\mathbb{F}_n$ groups.
\end{Thm}
\begin{proof}
The claim follows from the induced construction of dynamical systems with Green's imprimitivity theorem.
For the convenience of the reader, we give the precise construction.
Let $\Gamma$ and $\Lambda$ be as above.
Take an amenable minimal Cantor $\Lambda$-system $\gamma\colon\Lambda\curvearrowright X$ such that
its crossed product is a Kirchberg algebra in the UCT class
and its
$K$-theory $(K_0, [1]_0, K_1)$ is isomorphic to $(G_0, u, G_1)$.
On the space $\Gamma\times X$, define the equivalence relation $\sim_\Lambda$ by
\[(g, x)\sim_\Lambda (h, y)\Longleftrightarrow\exists k\in \Lambda, (h,y)=(gk^{-1}, \gamma(k)(x)).\]
Define the space $\Gamma\times_\Lambda X$ to be
the quotient space of $\Gamma\times X$ by the equivalence relation $\sim_\Lambda$.
It is easy to check that $\Gamma\times_\Lambda X$ is the Cantor set.
Define an action $\widetilde{\gamma}$ of $\Gamma$ on $\Gamma\times_\Lambda X$ by
$\widetilde{\gamma}(g)([h, x]):=[gh, x]$.
Here $[h, x]$ denotes the equivalence class of $(h, x)$ under $\sim_\Lambda$.
By the definition of $\widetilde{\gamma}$ (with the corresponding properties of $\gamma$),
we can check easily that $\widetilde{\gamma}$ is minimal, amenable, and topologically free.

Define $\pi\colon\Gamma\times_\Lambda X\rightarrow \Gamma/\Lambda$ by
$[h, x]\mapsto h\Lambda$.
Then by the definition of $\sim_\Lambda$, $\pi$ is a (well-defined) $\Gamma$-equivariant quotient map.
Notice that $\pi^{-1}(e\Lambda)$ is $\Lambda$-equivariantly homeomorphic to $X$.
Now applying Green's imprimitivity theorem to $\pi$,
we get the isomorphism
\[C(\Gamma\times_\Lambda X)\rtimes_{\widetilde{\gamma}}\Gamma\cong\mathbb{M}_k(C(X)\rtimes_{\gamma}\Lambda).\]
(An isomorphism can be given by a similar way to that in Theorem \ref{Thm:bou}.)
This shows that the Cantor $\Gamma$-system $\widetilde{\gamma}$
has the desired properties.
\end{proof}
Combining Theorems \ref{Thm:Inf} and \ref{Thm:Infv}, we obtain the following decomposition theorem.
\begin{Cor}\label{Cor:Dec}
For a torsion free abelian group $G$ of infinite rank,
consider a Kirchberg algebra $A$ in the UCT class satisfying
$(K_0(A), [1]_0, K_1(A))\cong(G\oplus\mathbb{Q}/\mathbb{Z}, 0, \mathbb{Z}^{\oplus\infty})$.
Then for any virtually free group $\Gamma$, $A$ is decomposed
as the crossed product of an amenable minimal topologically free Cantor $\Gamma$-system.
\end{Cor}
\begin{proof}
Let $G$ be as stated.
Thanks to Theorems \ref{Thm:Inf} and \ref{Thm:Infv}, it suffices to show that
there is an embedding $\iota\colon G\hookrightarrow\mathbb{Q}^{\oplus\infty}$ such that
its image $G'$ contains $\mathbb{Z}^{\oplus\infty}$ and
satisfies $N_{G'/\mathbb{Z}^{\oplus\infty}}=\prod_{p\in\mathcal{P}} p^{\infty}$.
To see this, take a maximal linear independent sequence $(x_n)_n$ of $G$.
Then the mapping $x_n\mapsto n^{-1}e_n\in\mathbb{Q}^{\oplus\infty}, n\in\mathbb{N}$
extends to the desired inclusion.
Here $(e_n)_n$ denotes the canonical basis of $\mathbb{Z}^{\oplus\infty}$.
\end{proof}
\begin{Rem}
The triple $(0, 0, \mathbb{Z}^{\oplus\infty})$ also satisfies the statement in Corollary \ref{Cor:Dec}.
In this case, each $\Gamma$-Cantor system can be taken to be free.
This follows from the proof of Theorem 5.1 in \cite{ES}, the existence of a minimal Cantor $\mathbb{F}_n$-system
whose $K_1$-group is isomorphic to $\mathbb{Z}^{\oplus\infty}$, and the Pimsner--Voiculescu exact sequence.
\end{Rem}
We say a Cantor system is profinite if it is of the form
$\varprojlim(\Gamma\curvearrowright \Gamma/\Gamma_m)_m$
for some (strictly) decreasing sequence $(\Gamma_m)_m$ of finite index subgroups of $\Gamma$.
When each $\Gamma_m$ is normal in $\Gamma$, the corresponding profinite Cantor system is free if and only if the intersection
$\bigcap_m \Gamma_m$ only consists of the unit element.
Our constructions in Theorems \ref{Thm:Inf} and \ref{Thm:Infv} also
provide continuously many amenable minimal free Cantor systems for every virtually free group.
\begin{Thm}\label{Thm:free}
Let $\Gamma$ be a virtually free group.
Then there exist continuously many amenable minimal free Cantor $\Gamma$-systems
whose crossed products are mutually non-isomorphic Kirchberg algebras in the UCT class.
\end{Thm}
\begin{proof}
We first show the assertion for the case $\Gamma=\mathbb{F}_n$.
For each nonempty subset $Q$ of $\mathcal{P}$, take a decreasing sequence
of finite index subgroups of $\mathbb{F}_n$ as follows.
First take an increasing sequence $(F_m)_m$ of finite subsets of $Q$ whose union is $Q$
and set $q_m:=\prod_{p\in F_m}p$ for each $m$.
Let $\pi_1\colon\mathbb{F}_n\rightarrow G_1$ be the quotient homomorphism
where $G_1$ is the quotient $\mathbb{F}_n^{\rm ab}/q_1\mathbb{F}_n^{\rm ab}$ of $\mathbb{F}_n$.
Then $\Gamma_1:=\ker\pi_1$ is a proper characteristic subgroup of $\mathbb{F}_n$ whose index
is a power of $q_1$.
Next consider the quotient homomorphism $\pi_2\colon\Gamma_1\rightarrow G_2$
where $G_2$ is the quotient $\Gamma_1^{\rm ab}/q_2\Gamma_1^{\rm ab}$ of $\Gamma_1$.
Then $\Gamma_2:=\ker\pi_2$ is a proper characteristic subgroup of $\Gamma_1$ whose index is a power of $q_2$.
Continuing this process inductively, we get a decreasing sequence
$(\Gamma_m)_m$ of subgroups of $\mathbb{F}_n$ that satisfies the following conditions.
\begin{itemize}
\item Each $\Gamma_m$ is a proper characteristic subgroup of $\Gamma_{m-1}$.
\item Each index $[\Gamma_m:\Gamma_{m-1}]$ is a power of $q_m$.
\end{itemize}
From the first condition, Levi's theorem \cite[Chap.1, Prop.3.3]{LS} implies that the intersection $\bigcap_m\Gamma_m$ only consists of the unit element.
Denote by $\alpha_Q$ the profinite Cantor system defined by the sequence $(\Gamma_m)_m$.
Then the proof of Theorem \ref{Thm:Inf} shows that the Cantor system $\gamma_Q:=\beta\times\alpha_Q$ is amenable and minimal,
the crossed product is a Kirchberg algebra in the UCT class,
and the torsion subgroup of $K_0$-group is isomorphic
to $\Lambda((n-1)q_\infty)$,
where $q_\infty:=\prod_{q\in Q}q^\infty$.
This completes the proof.
The case of virtually free groups is derived from the case of free groups
by the same method as that in the proof of Theorem \ref{Thm:Infv}.
\end{proof}
\begin{Rem}
In fact, the $K$-theory of Cantor systems in Theorem \ref{Thm:free} are computable.
For the case $\Gamma=\mathbb{F}_n$, the computation in Lemma \ref{Lem:Conn} shows
$$(K_0(C(X)\rtimes_{\gamma_Q} \mathbb{F}_n), [1]_0)\cong (R_Q^{\oplus\infty}\oplus\Lambda((n-1)q_\infty), 0\oplus [(n-1)^{-1}]),$$
where $R_Q$ is the subring of $\mathbb{Q}$ generated by $\{q^{-1}:q\in Q\}$ (and regard it as an additive group).
Computations for general cases then easily follows from the above computation.
Note that again by the skyscraper construction, we also can replace the position of the unit as before.
We also remark that by replacing the sequence $(q_m)_m$,
we can obtain more examples of amenable minimal free Cantor systems.
\end{Rem}

Next we give a more general construction of amenable minimal Cantor $\mathbb{F}_n$-systems.
This may provide more wild Cantor systems.
We show that their crossed products are Kirchberg algebras in the UCT class.

First, recall a few basic properties of torsion free elements
in hyperbolic groups.
For the free group case, it is not difficult to prove them directly.
\begin{Prop}[{\cite[Proposition 4.2]{KB}}]\label{Prop:Bou1}
Let $\Gamma$ be a hyperbolic group,
let $g$ be a torsion free element in $\Gamma$.
Then the boundary action of $g$ admits exactly two fixed points, say $\omega_+(g)$ and $\omega_-(g)$.
The points $\omega_+(g)$ and $\omega_-(g)$ are given by
$\lim_{n\rightarrow\infty} g^n$ and
$\lim_{n\rightarrow\infty} g^{-n}$ respectively in the Gromov compactification $\Gamma\cup\partial\Gamma$.
\end{Prop}
\begin{Prop}[{\cite[Theorem 4.3]{KB}}] \label{Prop:Bou2}
Let $\Gamma$ and $g$ be as above.
Then for any open subsets $U, V$ of $\partial\Gamma$ satisfying $\omega_+(g)\in U, \omega_-(g)\in V$, there exists
$n\in \mathbb{N}$ such that
$g^n(\partial\Gamma\setminus V)\subset U$.
\end{Prop}
\begin{Rem}\label{Rem:Bou}
In particular, for any $g\in\Gamma$ and $x\in\partial\Gamma\setminus\{\omega_-(g)\}$,
$\omega_+(g)$ is a limit point of the sequence $(g^n x)_{n=1}^\infty$
and similarly for $\omega_-$.
In fact, we have the genuine convergence, but we do not use this fact.
\end{Rem}

Now we establish a generalized construction.
\begin{Thm}\label{Thm:gmi}
Let $\alpha$ be the projective limit of a projective system $(\alpha_k)_k$ of non-faithful minimal dynamical systems 
of $\mathbb{F}_n$ such that
the underlying space of $\alpha_k$ is either finite or the Cantor set for each $k$.
Then the diagonal action $\beta\times\alpha$ is an amenable minimal Cantor $\mathbb{F}_n$-system
and its crossed product is a Kirchberg algebra in the UCT class.
\end{Thm}
\begin{proof}
Since all the properties we need to check are preserved by a projective limit,
it suffices to show the assertion for the case $\Gamma:=\ker(\alpha)\neq\{e\}$.

Amenability of $\beta\times\alpha$ is clear since it has an amenable quotient.
Then the crossed product satisfies the UCT by Tu's theorem \cite{Tu}.

Next we show minimality of $\beta\times\alpha$.
To see this, it suffices to show that the restriction
of $\beta$ to
$\Gamma$ is minimal.
Observe that for any
$g\in\mathbb{F}_n\setminus\{e\}$ and $h\in\mathbb{F}_n$,
we have $\omega_+(hgh^{-1})=h.\omega_+(g)$ and similarly for $\omega_-$.
This shows that the set $\{\omega_+(g):g\in\Gamma\setminus\{e\}\}$ is dense
in $\partial\mathbb{F}_n$.
Moreover, these equations show that for any $x\in\partial\mathbb{F}_n$,
there is $g\in\Gamma\setminus\{e\}$ such that both
$\omega_\pm(g)$ are not equal to $x$.
These observations with Remark \ref{Rem:Bou} show minimality of $\beta|_\Gamma$.

Finally, we prove pure infiniteness of the crossed product.
By the result of R\o rdam--Sierakowski \cite[Theorem 4.1]{RS},
it suffices to show that every projection in $C(\partial\mathbb{F}_n\times X)$ is properly infinite in the crossed product.
To see this, for a proper clopen subset $U$ of $\partial\mathbb{F}_n$,
take $g\in \Gamma\setminus \{e\}$ such that $\omega_-(g)\not\in U$.
Then by Proposition \ref{Prop:Bou2}, for any neighborhood
$V$ of $\omega_+(g)$, there is $n\in \mathbb{N}$ such that
$g^n.U\subset V$ holds.
Since $V$ can be chosen arbitrarily small
and $\Gamma$ acts minimally on $\partial\mathbb{F}_n$,
we see that for any nonempty clopen subset $V$ of $\partial\mathbb{F}_n$,
there is $g\in\Gamma$ such that
$g.U\subset V$ holds.
Thus, any projection of the form $\chi_U\otimes \chi_V$, 
where $U, V$ are clopen subsets of $\partial \mathbb{F}_n, X$,  respectively,
is properly infinite in $C(\partial\mathbb{F}_n\times X)\rtimes\mathbb{F}_n$.
Since every projection in $C(\partial\mathbb{F}_n \times X)$
can be written as a finite orthogonal sum of projections
of the above form,
the proof is now completed.
\end{proof}
\begin{Rem}
The above proof shows that for any Cantor system $\gamma$ as above,
its transformation groupoid $X\rtimes_\gamma\mathbb{F}_n$ (\cite[Example 5.6.3]{BO}) is purely infinite in the sense of Matui \cite[Definition 4.9]{Mat}.
Then, thanks to Matui's theorem \cite[Theorem 4.16]{Mat}, the commutator subgroup $D([[\gamma]])$ of the topological full group $[[\gamma]]$
turns out to be simple. (For the definition of the topological full group, see the next section.)
\end{Rem}
\begin{Rem}
For the proof of pure infiniteness of crossed products in Theorem \ref{Thm:gmi},
the theorem of Laca and Spielberg \cite[Theorem 9]{LaS} is also applicable.
In that case, we need to show that the action is a local boundary action.
This easily follows from our proof.
Moreover, in this case, the proof also works for any ICC hyperbolic groups.
This provides amenable minimal dynamical systems of
an ICC hyperbolic group $\Gamma$
on the product of $\partial\Gamma$ and the Cantor set whose
crossed products are Kirchberg algebras in the UCT class.
However, remark that the underlying spaces of these dynamical systems
are not the Cantor set unless $\Gamma$ is virtually free \cite[Theorem 8.1]{KB}.
\end{Rem}

The $K_1$-groups of these Cantor systems have the following simple formula.

\begin{Prop}
Let $\alpha\colon \mathbb{F}_n\curvearrowright X$ be a minimal $\mathbb{F}_n$-system as in Theorem $\ref{Thm:gmi}$.
Then the $K_1$-group of the diagonal action $\beta\times \alpha$ is
isomorphic to $\mathbb{Z}^{\oplus L_\alpha},$
where $L_\alpha$ is
given by $(n-1)\cdot[\mathbb{F}_n:\ker(\alpha)]+1$.
\end{Prop}
\begin{proof}
If each index $[\mathbb{F}_n : \ker(\alpha_m)]$ is finite, then the claim follows from the proof of Theorem \ref{Thm:Inf}.
So assume that for some $m$, the subgroup $\ker(\alpha_m)$ is of infinite index.
Since the $K_1$-group of $\beta\times \alpha$ is free abelian,
it suffices to show that the rank of $K_1$-group is infinite.
To see this, by the Pimsner--Voiculescu six term exact sequence, it suffices to show
that for any non-faithful Cantor system $\gamma \colon \mathbb{F}_n\curvearrowright X$,
its $K_1$-group is of infinite rank.
So let  $\gamma \colon \mathbb{F}_n\curvearrowright X$ be a non-faithful Cantor system.
Take a nonzero element $z$ of $\ker(\gamma)$.
For $1\leq j\leq |z|$, denote by $z^{(j)}$ the last $j$th segment of $z$
and denote by $a_j$ the first alphabet of $z^{(j)}$.
We put $z^{(0)}=e$ for convenience.
Put $F:=\{z^{(j)}\}_{0 \leq j\leq |z|}$.
Since $X$ is infinite, for any $N\in\mathbb{N}$, there are $N$ points $x_1, \ldots, x_N \in X$
such that the sets $F\cdot x_i$ are mutually disjoint and that the equality
$z^{(j)}.x_i=z^{(k)}.x_i$ implies the existence of a neighborhood $V$ of $x_i$
such that $z^{(j)}.x=z^{(k)}.x$ for all $x\in V$.
Then, for each $i$, we can take a clopen neighborhood $V_i$ of $x_i$ such that
the intersection $z^{(j)}.V_i\cap z^{(k)}.V_l$ is nonempty only if the equalities
$i=l$ and $z^{(j)}.V_i=z^{(k)}.V_l$ hold.
For each $k$ and $j$, define a clopen subset $V_{k, j}$ of $X$ by
\[V_{k, j}=\left\{ \begin{array}{ll}
z^{(j-1)}.V_k & {\rm if\ } a_j \in S \\
z^{(j)}.V_k & {\rm if\ } a_j \in S^{-1}.\\
\end{array} \right.\]
For each $k$, take $j(k),  l(k) \in \{1, 2, \ldots, |z|\}$ such that $j(k) \leq l(k)$,
$z^{(j(k)-1)}.V_k= z^{(l(k))}.V_k$, and
$z^{(l)}.V_k\neq z^{(m)}.V_k$
for any two distinct elements $l, m\in \{j(k), \ldots, l(k) \}$.
Then, for each $k$, set
\[p_k:=\sum_{j=j(k)}^{l(k)} \chi_{V_{k, j}}\otimes\delta_{a_j} \in C(X, \mathbb{Z})^{\oplus S}.\]
Here we put $\delta_{s^{-1}}:=-\delta_{s}$ for $s\in S$.
We show that $p_k$ is nonzero for each $k$.
This is trivial if $l(k)= j(k)$.
If $l(k)= j(k)+1$, then either $V_{k, j(k)}\neq V_{k, l(k)}$
or  $\chi_{V_{k, j(k)}}\otimes\delta_{a_{j(k)}} =\chi_{V_{k, l(k)}}\otimes\delta_{a_{l(k)}}$ holds.
This shows $p_k\neq 0$.
If $l(k)> j(k)+1$,
then by the choices of $V_k$, $j(k)$, and $l(k)$,
for any $j\in \{ j(k)+1, \ldots, l(k)-1\} (\neq \emptyset)$,
we have
$\chi_{V_{k, l}}\otimes \delta_{a_l} \neq \pm \chi_{V_{k, j}}\otimes \delta_{a_j}$ for all $l \in \{j(k), \ldots, l(k)\}\setminus\{j\}$.
Since the set $\{\chi_{V_{k, j}}\otimes \delta_{s}:j=1,\ldots, |z|, s\in S\}$ is linearly independent,
we have $p_k\neq 0$.

By disjointness of $(\bigcup_j V_{k, j})$, the sequence $(p_k)_k$ is linearly independent.
It is easy to check that each $p_k$ is contained in the kernel of
the connecting map 
\[\eta_{\gamma} \colon C(X)^{\oplus S}\rightarrow C(X)\]
of the Pimsner--Voiculescu six term exact sequence.
This shows that the rank of $K_1(C(X)\rtimes_{\gamma}  \mathbb{F}_n)$ is not less than $N$.
Since $N$ is arbitrary, the rank must be infinite.
\end{proof}

However, computations of $K_0$-groups of these dynamical systems seem to be difficult in general.
\begin{Prob}
Find a formula of the $K_0$-group of $\beta\times\alpha$ in Theorem \ref{Thm:gmi}
by using certain (computable) invariants of $\alpha$.
It would be already interesting to find a formula for $K_0^{\rm tor}$ or
$K_0/K_0^{\rm tor}$.
\end{Prob}
Note that if $\alpha$ is a profinite Cantor system, then the torsion part of
the $K_0$-group of $\beta\times\alpha$ and the position of
the unit are easily determined by the same method as that in the proof of Theorem \ref{Thm:Inf}.
If we further assume that $\alpha$ is defined by the decreasing sequence $(\Gamma_m)_m$
of finite index subgroups of $\mathbb{F}_n$ such that for all sufficiently large $m$,
$\Gamma_m$ is a normal subgroup of $\Gamma_{m-1}$ and the quotient $\Gamma_{m}/\Gamma_{m-1}$ is solvable,
then the $K_0$-group of $\beta\times\alpha$ is also computable and the invariant
$(K_0, [1]_0, K_1)$ is isomorphic to one of those appearing in Theorem \ref{Thm:Inf} with $k=1$.

We end the section by giving the analogue of Theorem \ref{Thm:free} for virtually $\mathbb{F}_\infty$ groups.
Here $\mathbb{F}_\infty$ denotes the free group on countably infinite generators.
\begin{Lem}\label{Lem:Finf}
Let $D(\mathbb{F}_2)$ be the commutator subgroup of $\mathbb{F}_2$.
Then the restriction of the boundary action to $D(\mathbb{F}_2)$ is amenable, minimal, and
its crossed product satisfies the following properties.

\begin{itemize}
\item It is a Kirchberg algebra in the UCT class.
\item The unit $[1]_0$ generates a subgroup isomorphic to $\mathbb{Z}$.
\item For any $n\geq 2, [1]_0\not\in n K_0(A)$.
\end{itemize}
The same statement also holds for any finite index subgroup of $D(\mathbb{F}_2)$.
\end{Lem}
\begin{proof}
Amenability of $\beta|_{D(\mathbb{F}_2)}$ is clear.
The proof of Theorem \ref{Thm:gmi} shows minimality of $\beta_2|_{D(\mathbb{F}_2)}$ and
pure infiniteness of the crossed product.
For any $n\geq 3$, there is a finite index subgroup $\Lambda_n$ of $\mathbb{F}_2$
which contains $D(\mathbb{F}_2)$ and is isomorphic to $\mathbb{F}_n$.
(This follows from Schreier's formula, for example.)
Note that for such $\Lambda_n$, the restriction $\beta_2|_{\Lambda_n}$ is isomorphic to the boundary action of $\mathbb{F}_n$.
Hence for any $n\geq 2$, there is a unital embedding
$C(\partial \mathbb{F}_2)\rtimes D(\mathbb{F}_2) \rightarrow C(\partial \mathbb{F}_n)\rtimes \mathbb{F}_n$.
This shows that for any $n\geq 2$,
there is a group homomorphism
$K_0(C(\partial \mathbb{F}_2)\rtimes D(\mathbb{F}_2))\rightarrow \mathbb{Z}^n\oplus\mathbb{Z}_{n-1}$
that maps the unit $[1]_0$ to the canonical generator of $\mathbb{Z}_{n-1}$.
This shows the claim for $D(\mathbb{F}_2)$.

Now let a finite index subgroup $\Lambda$ of $D(\mathbb{F}_2)$ be given.
It suffices to show minimality of $\beta_2|_{\Lambda}$ and pure infiniteness of the crossed product.
Again by the proof of Theorem \ref{Thm:gmi}, it suffices to show that $\Lambda$ contains a nontrivial normal subgroup
of $\mathbb{F}_2$.
To see this, consider the group action
$D(\mathbb{F}_2) \curvearrowright \bigsqcup_{g\in \mathbb{F}_2} D(\mathbb{F}_2) /g\Lambda g^{-1}$
given by the left multiplication action on each component.
Then the kernel of the action is a normal subgroup of $\mathbb{F}_2$ contained in $\Lambda$.
Moreover, it must be nontrivial because $D(\mathbb{F}_2)$ is not a torsion group.
\end{proof}
\begin{Rem}
By the Pimsner--Voiculescu six term exact sequence for free groups,
in fact the above group homomorphisms can be chosen to be surjective.
Hence $K_0(C(\partial \mathbb{F}_2)\rtimes\Lambda)$ is not finitely generated
for any subgroup $\Lambda$ of $D(\mathbb{F}_2)$.
\end{Rem}
\begin{Thm}
Let $\Gamma$ be a virtually $\mathbb{F}_\infty$ group.
Then there exist continuously many amenable minimal free Cantor $\Gamma$-systems
whose crossed products are mutually non-isomorphic Kirchberg algebras in the UCT class.
\end{Thm}
\begin{proof}
By the induced dynamical system construction,
it suffices to show the claim for $\Gamma=\mathbb{F}_\infty$.
For any nonempty set $Q$ of prime numbers,
take a sequence $(q_n)_n$ of elements of $Q$
satisfying $\{ q_n\}_n=Q$.
Then take a decreasing sequence $(\Lambda_n)_n$ of finite index normal subgroups of $\mathbb{F}_\infty$
that satisfies $\Lambda_1=\mathbb{F}_\infty$ and $[\Lambda_n: \Lambda_{n+1}]=q_n$ for $n\in \mathbb{N}$.
We identify $\mathbb{F}_\infty$ with the commutator subgroup of $\mathbb{F}_2$ by a fixed isomorphism \cite[Chap.1 Prop.3.12]{LS}.
Take a decreasing sequence $(\Gamma_{n})_n$ of finite index normal subgroups of $\mathbb{F}_2$
such that the index $[\Gamma_n : \Gamma_{n+1}]$ is a power of $q_n$ and $\bigcap_n\Gamma_n=\{1\}$ holds.
(Such a sequence can be constructed in the same way as that in the proof of Theorem \ref{Thm:free}.)
Now set $\Upsilon_n:=\Gamma_n\cap \Lambda_n$.
Then $(\Upsilon_n)_n$ is a decreasing sequence of finite index normal subgroups of $\mathbb{F}_\infty$
such that the index $[\mathbb{F}_\infty : \Upsilon_n]$ divides a power of $q_1\cdots q_{n-1}$ and
is divisible by $q_1\cdots q_{n-1}$ for each $n$,
and $\bigcap_n\Upsilon_n=\{1\}$ holds.
(The first claim follows from the existence of a group embedding $\mathbb{F}_\infty/\Upsilon_n \rightarrow (\mathbb{F}_\infty/\Lambda_n) \times (\mathbb{F}_2/\Gamma_n)$.)
Now consider the Cantor system
$\gamma_Q:=\varprojlim((\beta_2|_{\mathbb{F}_\infty})\times \alpha_n\colon \mathbb{F}_\infty\curvearrowright \partial \mathbb{F}_2\times (\mathbb{F}_\infty/\Upsilon_n)).$
Then Lemma \ref{Lem:Finf} shows that the Cantor system
$\gamma_Q$ is amenable, minimal, free, and its crossed product
is a Kirchberg algebra in the UCT class.
Furthermore, a similar argument to that in the proof of Theorem \ref{Thm:Inf} shows the equality
\[\{p \in \mathcal{P} : [1]_0\in p K_0(C(X)\rtimes_{\gamma_Q}\mathbb{F}_\infty)\}=Q.\]
This shows that the crossed products of $\gamma_Q$'s are mutually non-isomorphic.
\end{proof}

\section{Classification of diagonal actions of boundary actions and products of odometer transformations}\label{sec:odo}
In this section, using the technique of computation of $K$-theory developed in Section \ref{sec:gmi},
we classify the amenable minimal Cantor $\mathbb{F}_n$-systems given by
the diagonal actions of the boundary actions and the products of the odometer transformations.

First we recall the definition of the odometer transformation.
For an infinite supernatural number $N$,
take a sequence $(k_m)_m$ of natural numbers such that
$k_m|k_{m+1}$ for all $m$ and ${\rm lcm}\{k_m:m\}=N$.
The odometer transformation of type $N$ is then defined as
the projective limit of the projective system
$(\mathbb{Z}\curvearrowright \mathbb{Z}_{k_m})_m$.
In this paper, we denote it by $\alpha_N$.
(Obviously, the definition of $\alpha_N$ only depends on $N$.)

Let $2\leq n<\infty$, let $1\leq k \leq n$,
and let $N_1, \ldots, N_k$ be a sequence of infinite supernatural numbers.
Fix an enumeration $\{s_1, \ldots, s_n\}$ of $S(\subset \mathbb{F}_n)$.
Then define a Cantor $\mathbb{F}_n$-system by
\[\gamma_{N_1,\ldots, N_k}^{(n)}:=\beta_n\times\left(\prod_{j=1}^k\alpha_{N_j}\circ \pi_j^{(n)}\right),\]
where $\pi_j^{(n)}$ denotes the homomorphism
$\pi_j^{(n)}\colon\mathbb{F}_n\rightarrow \mathbb{Z}$
given by $s_j\mapsto 1$ and $s_i \mapsto 0$ for $i\neq j$
for each $j$.
By the result of the previous section, each $\gamma_{N_1,\ldots, N_k}^{(n)}$
is an amenable minimal Cantor $\mathbb{F}_n$-system
and similar computations to those in Lemmas \ref{Lem:auto} and \ref{Lem:Conn} and
Theorem \ref{Thm:Inf} show the following theorem.
\begin{Thm}\label{Thm:odo}
Let $\gamma_{N_1,\ldots, N_k}^{(n)}$ be as above.
Then the crossed product of $\gamma_{N_1,\ldots, N_k}^{(n)}$ satisfies the following conditions.
\begin{itemize}
\item The pair of $K_0$-group and the unit $[1]_0$ is isomorphic to 
\[\left(\left(\bigoplus_{i=1}^k\Upsilon(N_i)\right)\oplus\mathbb{Z}^{\oplus \infty}\oplus\Lambda((n-1)N_1\cdots N_k), 0\oplus 0\oplus [(n-1)^{-1}]\right).\]
\item The $K_1$-group is isomorphic to $\mathbb{Z}^{\oplus \infty}$.
\item It is a Kirchberg algebra in the UCT class.
\end{itemize}
\end{Thm}
\begin{proof}
For each $j=1, \ldots, k$, take a sequence
$(n({m, j}))_m$ of natural numbers that satisfies the equation
$\prod_m n(m, j)=N_j$.
We further assume that for each $m$, only one $j$, say $j_m$, satisfies $n(m, j)\neq 1$.
Put $N(m, j):=n(1, j)\cdots n(m, j)$ and $M(m):=N(m-1, j_m)$.
Then for each $m$, consider the surjective homomorphism
$q_m\colon\mathbb{F}_n\rightarrow \bigoplus_{j=1}^k\mathbb{Z}/N(m, j)$ defined by mapping $s_j$ to the canonical generator of $j$th direct summand for $j=1, \ldots, k$
and the other $s\in S$ to $0$.
Set $\Gamma_m:=\ker(q_m).$
Then the sequence $(\Gamma_m)_m$ defines a profinite Cantor system $\alpha$.
By definition we have $\gamma_{N_1,\ldots, N_k}^{(n)}=\beta\times\alpha$.

Next we inductively choose suitable free bases of $\Gamma_m$'s as follows.
First set $W_0:=S$ and $N(0, j)=1$ for convenience.
Then define $W_m$ by
\[W_m:=\left(W_{m-1}\setminus \{s_{j_m}^{M(m)}\}\right)\cup\{s_{j_m}^{N(m, j_m)}\}\cup Z_m,\]
where
\[Z_m:=\left\{w^{-1}s_{j_m}^{lM(m)}ws_{j_m}^{-lM(m)}: w\in W_{m-1}\setminus \{s_{j_m}^{M(m)}\}, 1\leq l< n(j_m, m)\right\}.\]
It is easy to check that for each $m$, the set $W_m$
is a free basis of $\Gamma_m$. 

Combining the computations used in the proofs of Lemmas \ref{Lem:auto} and \ref{Lem:Conn}, we can show that
the free part of the $K_0$-map induced from the canonical inclusion
\[C(\partial\mathbb{F}_n\times(\mathbb{F}_n/\Gamma_{m-1}))\rtimes\mathbb{F}_n\rightarrow C(\partial\mathbb{F}_n\times(\mathbb{F}_n/\Gamma_{m}))\rtimes\mathbb{F}_n\]
is given by
\[[q_{({s_{j_m}^{N(m-1, j_m)}};W_{m-1})}]_0\mapsto n(m, j_m)[q_{({s_{j_m}^{N(m, j_m)}}; W_m)}]_0\]
and
\[[q_{(t, W_{m-1})}]_0\mapsto [q_{(t; W_m)}]_0 {\rm\ for\ }t\in W_{m-1}\setminus\{s_{j_m}^{N(m-1, j_m)}\}.\]
Now the proof of Theorem \ref{Thm:Inf} completes the computation.
\end{proof}

The invariants appearing in Theorem \ref{Thm:odo} are completely classified
in terms of $(n; N_1,\ldots, N_k)$ as follows.
A supernatural number $N$ is recovered from the group $\Lambda(N)$ by the formula
$N={\rm lcm}\{{\rm ord}(x):x\in\Lambda(N)\setminus\{ e\}\}.$
On the other hand, from the group $G=\left(\bigoplus_{i=1}^k\Upsilon(N_i)\right)\oplus\mathbb{Z}^{\oplus \infty}$($\cong K_0/K_0^{\rm tor}$),
we can recover the subgroup $\bigoplus_{i=1}^k\Upsilon(N_i)$
as the subgroup generated by the subset of all elements $x$ such that the set
$\{n\in\mathbb{N}:{\rm there\ exists\ } y \in G {\rm \ with\ } ny=x\}$ is infinite.
Then it is known that the two groups $\bigoplus_{i=1}^k\Upsilon(N_i)$ and
$\bigoplus_{i=1}^m\Upsilon(M_i)$ are isomorphic if and only if
$k=m$ and there are a permutation $\sigma\in\mathfrak{S}_k$ and natural numbers $n_1, \ldots, n_k$,
$m_1, \ldots, m_k$ such that
$n_iN_i=m_iM_{\sigma(i)}$ holds for all $i$.
(This follows from Baer's theorem \cite[Vol.II, Prop.86.1]{Fuc} and the isomorphism condition of groups $\Upsilon(M)$.)

On the set of all finite sequences of infinite supernatural numbers,
we define the equivalence relation $\sim$ as follows.
For two finite sequences $(N_1, \ldots, N_k)$ and $(M_1, \ldots, M_l)$,
we say the relation $\sim$ holds if
$k=l$ and there are a permutation $\sigma\in\mathfrak{S}_k$ and natural numbers $n_1, \ldots, n_k$,
$m_1, \ldots, m_k$ such that $\prod_{i=1}^k n_i=\prod_{i=1}^k m_i$
and $n_iN_i=m_iM_{\sigma(i)}$ hold for all $i$.
Denote by $[N_1,\ldots, N_k]$ the equivalence class of $(N_1, \ldots, N_k)$ under $\sim$.
From the above observations, the equivalence class $[N_1, \ldots, N_k]$
is a complete invariant of the group $\left(\bigoplus_{i=1}^k\Upsilon(N_i)\right)\oplus\mathbb{Z}^{\oplus \infty}\oplus\Lambda(N_1\cdots N_k)$.
Here we record it as a proposition.
\begin{Prop}\label{Prop:grp}
For a sequence $N_1, \ldots, N_k$ of infinite supernatural numbers,
define the group $G(N_1, \ldots, N_k)$ by
\[\left(\bigoplus_{i=1}^k\Upsilon(N_i)\right)\oplus\mathbb{Z}^{\oplus \infty}\oplus\Lambda(N_1\cdots N_k).\]
Then two groups $G(N_1, \ldots, N_k)$ and $G(M_1, \ldots, M_l)$ are isomorphic
if and only if $[N_1, \ldots, N_k]=[M_1, \ldots, M_m]$,
where $[\cdot]$ denotes the equivalence class of the equivalence relation $\sim$ defined in the above.
In particular, for two free groups $\mathbb{F}_n$, $\mathbb{F}_{m}$ and for two finite sequences of infinite supernatural numbers $N_1, \ldots, N_k$, $M_1,\ldots, M_{l}$ with $k\leq n$ and $l\leq m$,
the pairs $(K_0, [1]_0)$ of the corresponding two $\gamma$'s are isomorphic if and only if
$n=m$ and $[N_1, \ldots, N_k]=[M_1,\ldots, M_{l}]$ hold.
\end{Prop}
From Proposition \ref{Prop:grp} and Matui's theorem \cite{Mat} with a little extra effort, we can classify the strong orbit equivalence classes,
the topological full groups, the crossed products, and the continuous orbit equivalence classes of $\gamma_{N_1,\ldots, N_k}^{(n)}$'s.

To state these definitions, first we recall the definition of orbit cocycles.
\begin{Def}\label{Def:coc}
Let $\alpha_1, \alpha_2$ be minimal topologically free Cantor systems of groups $\Gamma_1$, $\Gamma_2$,
respectively.
Let $F \colon X_1\rightarrow X_2$ be an orbit preserving homeomorphism between $\alpha_1$ and $\alpha_2$.
A map $c\colon \Gamma_1\times X_1\rightarrow \Gamma_2$
is said to be an orbit cocycle of $F$
if it satisfies the equation $F(\alpha_1(g)(x))=\alpha_2(c(g, x))(F(x))$ for all $(g, x)\in \Gamma_1\times X_1$.
\end{Def}
Note that by topological freeness,
the cocycle equation $$c(g, h.x)c(h, x)=c(gh, x)$$
holds on a dense subset of $\Gamma_1\times\Gamma_1\times X_1$.
If we further assume that either $\alpha_2$ is free or $c$ is continuous,
then the cocycle equation holds on $\Gamma_1\times\Gamma_1\times X_1$.

\begin{Def}\label{Def:COE}
Let $\alpha_1$, $\alpha_2$, $\Gamma_1$, and $\Gamma_2$ be as above.
Two Cantor systems $\alpha_1$ and $\alpha_2$ are said to be continuously orbit equivalent
if there exists an orbit preserving homeomorphism $F \colon X_1\rightarrow X_2$ such that
both $F$ and $F^{-1}$ admit a continuous orbit cocycle.
\end{Def}
It is easy to check that two minimal topologically free Cantor systems are continuously orbit equivalent
if and only if their transformation groupoids are isomorphic as $\acute{{\rm e}}$tale groupoids.

Next we recall the definition of the topological full group.
This is the group that gathers the local behaviors of a topologically free Cantor system.
\begin{Def}
The topological full group $[[\gamma]]$ of a topologically free Cantor system $\gamma$ is
the group of all homeomorphisms $F$ on $X$ with the property that
for each $x\in X$, there exist a neighborhood $U$ of $x$ and
$s\in\Gamma$ that satisfy $F(y)=s.y$ for all $y\in U$.
\end{Def}

Clearly, the continuous orbit equivalence implies the isomorphism of topological full groups.
With minimality assumption, H. Matui shows the converse also holds.

\begin{Thm}[A special case of {\cite[Theorem 3.10]{Mat}}]\label{Thm:Mat}
Let $\gamma_i$ be a minimal topologically free Cantor system of a group $\Gamma_i$ for $i=1, 2$.
Then the following are equivalent.
\begin{enumerate}[\upshape (1)]
\item They are continuously orbit equivalent.
\item Their topological full groups are isomorphic $($as discrete groups$)$.
\item The commutator subgroups of their topological full groups are isomorphic $($as discrete groups$)$.
\end{enumerate}
\end{Thm}

We next introduce a notion of strong orbit equivalence of Cantor systems for general groups as follows.
\begin{Def}\label{Def:SOE}
We define the relation $R$ on the class of Cantor systems as follows.
For two Cantor systems $\alpha_i\colon \Gamma_i\curvearrowright X_i$, $i=1, 2$.
we declare the relation $R(\alpha_1, \alpha_2)$ holds if the following conditions hold.
There is an orbit preserving homeomorphism $F \colon X_1\rightarrow X_2$ and a generating set $S_i$ of $\Gamma_i$ for $i=1, 2$ that admit an orbit cocycle $c$ of $F$
with the property that
for each $s\in S_1$, the restriction of $c$ on $\{s\}\times X_1$
has at most one point of discontinuity,
and the same condition also holds when we replace $X_1$ by $X_2$, $F$ by $F^{-1}$, and $S_1$ by $S_2$.
Unfortunately, the relation $R$ seems not to satisfy the transitivity.
(This is in fact an equivalence relation if we only consider the minimal
Cantor $\mathbb{Z}$-systems.
This is already highly nontrivial; this is a consequence of a result in \cite{GPS}.)
For this reason, we define the equivalence relation $\sim$ to be the one generated by $R$,
and say $\alpha_1$ is strongly orbit equivalent to $\alpha_2$
if $\alpha_1\sim\alpha_2$ holds.
\end{Def}
\begin{Exm}
Let $\alpha_k$ and $\beta_k$ be Cantor $\mathbb{Z}$-systems for $k=1,\ldots, n$.
Assume that the underlying spaces of $\alpha$'s are the same one, say $X$,
and similarly for that of $\beta$'s, say $Y$.
Suppose that the relations $R(\alpha_k, \beta_k)$, $k=1, 2, \ldots, n$ are implemented by the same homeomorphism
$F \colon X\rightarrow Y$.
Then the free product dynamical systems \lower0.25ex\hbox{\LARGE $\ast$}${}_k\alpha_k$ and \lower0.25ex\hbox{\LARGE $\ast$}${}_k\beta_k$
are strongly orbit equivalent in the sense of Definition \ref{Def:SOE}.
\end{Exm}
Before completing the classification, we need a lemma about the strong orbit equivalence,
which may be well-known to specialists.
This claims that for two Cantor systems of free groups, the isomorphism of $K_0$-invariants is
a necessary condition for strong orbit equivalence. The idea of the proof comes from \cite{GPS}.
\begin{Lem}\label{Lem:SOE}
Let $\gamma_i\colon\mathbb{F}_{n_i}\curvearrowright X_i$ be topologically free Cantor systems with $2\leq n_i<\infty$ for $i=1,2$.
Assume that $\gamma_1$ is strongly orbit equivalent to $\gamma_2$.
Then their $K_0$-invariants $(K_0(C(X_i)\rtimes_{\gamma_i} \mathbb{F}_{n_i}), [1]_0); i=1, 2$ are isomorphic.
\end{Lem}
\begin{proof}
We may assume that the equalities $X_1=X_2=X$ hold and the identity map is an orbit preserving homeomorphism
that has orbit cocycles
each of which has discontinuous points at most one on each element of some generating sets $S_i$ of $\Gamma_i$.
By the Pimsner--Voiculescu six term exact sequence for free groups \cite{PV}, we obtain the isomorphism
\[K_0(C(X)\rtimes_{\gamma_i} \mathbb{F}_{n_i})\cong C(X,\mathbb{Z})/N_i,\]
where $N_i$ is the subgroup of $C(X,\mathbb{Z})$ generated by
elements of the form $\chi_E-\chi_{\gamma_i(s)(E)}$ for clopen subsets $E$ of $X$ and $s\in\mathbb{F}_{n_i}$.
Note that under the isomorphism, the unit $[1]_0$ is mapped to $1_X+N_i$.

From the above isomorphism, our claim follows from the equation $N_1=N_2$.
By the symmetry of the argument, it is enough to show that the inclusion $N_1\subset N_2$ holds.
To see this, let $E$ be a clopen subset and $s\in S_1$.
Replacing $E$ be $X\setminus E$ if necessarily, which does not change the difference $\chi_E -\chi_{\gamma_1(s)(E)}$ up to sign,
we may assume that there is an orbit cocycle $c$ that is continuous on $\{s\}\times E$.
Define $c_s(x):=c(s, x)$ for $s\in \Gamma_1$ and $x\in X$.
Set $\mathfrak{F}:=c_s(E)$, which is finite by the continuity assumption.
Then we have
\[\gamma_1(s)(E)=\bigsqcup_{g\in \mathfrak{F}} \gamma_2(g)(c_s^{-1}(\{g\})\cap E).\]
This shows \[\chi_E-\chi_{\gamma_1(s)(E)}=\sum_{g \in \mathfrak{F}}(\chi_{c_s^{-1}(\{g\})\cap E}-\chi_{\gamma_2(g)(c_s^{-1}(\{g\})\cap E)})\in N_2.\]
Since $S_1$ generates $\Gamma_1$,
we obtain the inclusion $N_1\subset N_2$.
\end{proof}
Now we give the classification results for $\gamma$'s.
\begin{Thm}\label{Thm:Cla}
Let $\gamma_1=\gamma_{N_1,\ldots, N_k}^{(n)}$ and $\gamma_2=\gamma_{M_1,\ldots, M_l}^{(m)}$ be as before.
Then the following conditions are equivalent.
\begin{enumerate}[\upshape (1)]
\item They are strongly orbit equivalent.
\item They are continuously orbit equivalent.
\item Their topological full groups are isomorphic.
\item The commutator subgroups of their topological full groups are isomorphic.
\item Their crossed products are isomorphic.
\item Their $K_0$-invariants $(K_0, [1]_0)$ are isomorphic.
\end{enumerate}
\end{Thm}
\begin{proof}
The implications (2) $\Rightarrow$ (3) $\Rightarrow$ (4), (2) $\Rightarrow$ (1) and (2) $\Rightarrow$ (5) $\Rightarrow$ (6) are clear.
The implication (4) $\Rightarrow$ (2) follows from Theorem \ref{Thm:Mat} and
the implication (1) $\Rightarrow$ (6) follows from Lemma \ref{Lem:SOE}.
Now it is left to prove the implication (6) $\Rightarrow$ (2).

Assume condition (6) holds.
Then by Proposition \ref{Prop:grp}, the equalities
$n=m$ and $k=l$ hold and there are a permutation $\sigma\in\mathfrak{S}_k$
and two sequences $n_1, \ldots, n_k$ and $m_1, \ldots, m_k$ of natural numbers
that satisfy $\prod_j n_j=\prod m_j$ and $n_jN_j=m_jM_{\sigma(j)}$ for all $j$.
By conjugating an automorphism of the free group, we may assume $\sigma$ is trivial.
Since the continuous orbit equivalence is an equivalence relation,
we further assume that there are a sequence of infinite supernatural numbers
$L_1, \ldots, L_k$ and a natural number $l$ satisfying
$N_1=lL_1$, $N_j=L_j$ for $j\neq 1$,
$M_2=lL_2$, and $M_j=L_j$ for $j \neq 2$.
Denote by $\Lambda_i$ the kernel of the surjection
$\rho_i:=q\circ\pi_i$ for $i=1, 2$.
Here $q$ denotes the quotient homomorphism
from $\mathbb{Z}$ onto $\mathbb{Z}_l$.
By the definition of $\gamma$'s, for $i=1, 2$, we have an $\mathbb{F}_n$-equivariant quotient map
$p_i\colon X_i\rightarrow\mathbb{F}_n/\Lambda_i$.
Here $X_i$ denotes the underlying space of $\gamma_i$ for $i=1, 2$.
Then, with the notion $Y_i:=p_i^{-1}(\Lambda_i),$
the homeomorphism $F_i$ from $X_i=\bigsqcup_{j=0}^{l-1} s_i^jY_i$ onto $Y_i\times\mathbb{Z}_l$
given by $x=s_i^jy\in s_i^{j}Y_i\mapsto (y, [j])$ shows that $\gamma_i$ is continuously orbit equivalent
to the Cantor system
$\tilde{\gamma_i}\boxtimes\lambda\colon \Lambda_i\times\mathbb{Z}_l\curvearrowright Y_i\times \mathbb{Z}_l$.
Here $\tilde{\gamma_i}$ denotes the restriction of $\gamma|_{\Lambda_i}$ to the $\Lambda_i$-invariant subspace $Y_i$ of $X_i$, $\lambda$ denotes the left translation action of $\mathbb{Z}_l$ on itself,
and the symbol `$\boxtimes$' stands for the product action.
From this, it suffices to show that $\tilde{\gamma_1}$ and $\tilde{\gamma_2}$ are continuously orbit equivalent.
Notice that for $i=1, 2$, the set
\[T_i:=\left\{s_i^l, t, s_i^jts_i^{-j}t^{-1}: t\in S\setminus\{s_i\}, 1\leq j\leq l-1\right\}\]
is a free basis of $\Lambda_i$.
Set $r:=\sharp T_1=\sharp T_2=l(n-1)+1$.
Then the isomorphism $\Lambda\cong\mathbb{F}_r$ given by the free basis $T_i$
shows that the dynamical system $\tilde{\gamma_i}$ is conjugate to $\gamma_{L_1,\ldots, L_k}^{(r)}$.
Thus $\tilde{\gamma_1}$ is continuously orbit equivalent to $\tilde{\gamma_2}$.
\end{proof}
\begin{Rem}
In fact, Matui's theorem \cite[Theorem 3.9]{Mat} shows stronger conclusion.
The six conditions in Theorem \ref{Thm:Cla}
are also equivalent to the following condition.
\begin{enumerate}
\setcounter{enumi}{6}
\item There exist intermediate subgroups $D([[\gamma_i]])\subset\Gamma_i\subset[[\gamma_i]]$ ($i=1, 2$)
such that $\Gamma_1$ is isomorphic to $\Gamma_2$.
\end{enumerate}
\end{Rem}
Finally we end the section by showing that for any profinite Cantor system $\alpha$,
the group $[[\beta\times\alpha]]$ has the Haagerup property
and is not finitely generated. This follows from Matui's result \cite[Theorem 6.7]{Mat} with a little extra effort.
\begin{Prop}\label{Prop:Haa}
For any profinite Cantor $\mathbb{F}_n$-system $\alpha$,
the topological full group of the diagonal action $\beta\times\alpha$ has the Haagerup property
and is not finitely generated.
More generally, for any intermediate subgroup of the inclusion $D([[\beta\times\alpha]])\subset[[\beta\times\alpha]]$,
the same statement holds true.
\end{Prop}
\begin{proof}
Take a projective system $(\alpha_m)_m$ of finite $\mathbb{F}_n$-systems
whose projective limit is $\alpha$.
Denote by $X_m$, $X$ the underlying space of $\alpha_m$, $\alpha$, respectively.
We may assume that all the connecting maps of the projective system are not injective.

By Remark \ref{Rem:Cun} and the same reason as in \cite[Sections 6.7.4 and 6.7.5]{Mat}, for each $m$,
the transformation groupoid of $\beta\times\alpha_m$ is isomorphic to the associated groupoid of a shift of finite type (determined by the matrix given in Remark \ref{Rem:Cun}).
Hence the group $[[\beta\times\alpha_m]]$ has the Haagerup property by \cite[Theorem 6.7]{Mat}.
For a topologically free Cantor $\Gamma$-system $\gamma$ acting on the Cantor set $Y$, set
\[P(\gamma):=\left\{(U_s)_{s\in \Gamma}: U_s{\rm\ is\ clopen} {\rm \ for\ all\ } s\in\Gamma, Y=\bigsqcup_{s\in \Gamma} U_s=\bigsqcup_{s\in \Gamma} s.U_s\right\}.\]
Then there is a bijective correspondence between $P(\gamma)$ and $[[\gamma]]$ given by 
\[(U_s)_{s\in \Gamma} \longleftrightarrow \ad(\sum_{s\in\Gamma}\lambda_s\chi_{U_s})\Big|_{C(Y)}.\]
For each $m$, the canonical quotient map from $X$ onto $X_m$ induces the inclusion
from $P(\beta\times\alpha_m)$ into $P(\beta\times\alpha)$.
This with the above correspondence gives a group inclusion $\iota_m$ from $[[\beta\times\alpha_m]]$ into $[[\beta\times\alpha]]$.
Denote the image of $\iota_m$ by $G_m$.
Then, from the corresponding properties of $P(\beta\times\alpha_m)$'s,
the sequence $(G_m)_m$
is strictly increasing and its union exhausts $[[\beta\times\alpha]]$.
This shows that the group $[[\beta\times\alpha]]$ has the Haagerup property and is not finitely generated.

For intermediate subgroups, the Haagerup property is clear since
it passes to subgroups.
To show that the intermediate subgroups are not finitely generated,
by passing to a subsequence, we may assume that for each $m$,
the fraction $(\sharp X_m)/(\sharp X_{m-1})$ is greater than $2$.
Then from the fact $D(\mathfrak{S}_k)\neq\{e\}$ for $k\geq 3$,
we can check easily that for each $m$, the set $D(G_m)\setminus G_{m-1}$
is nonempty. This ends the proof.
\end{proof}

\subsection*{Acknowledgement}
The author would like to thank Masaki Izumi, Hiroki Matui, Narutaka Ozawa, and Mikael R\o rdam
for fruitful discussions and comments.
He also thanks to Yasuyuki Kawahigashi for advice on the manuscript.
He is supported by Research Fellow of the Japan Society for the Promotion
of Science (No.25-7810) and the Program of Leading Graduate Schools, MEXT, Japan.

\end{document}